\documentclass[10pt,a4paper,reqno]{amsart}

\usepackage[utf8]{inputenc}
\usepackage[T1]{fontenc}
\usepackage[english]{babel}
\usepackage{amsmath}
\usepackage{amsfonts}
\usepackage{amssymb}
\usepackage{amsthm}
\usepackage{ucs}
\usepackage{graphicx}
\usepackage{xcolor}
\usepackage{dsfont}
\usepackage{tikz}
\usepackage{wasysym}
\usepackage{physics}
\usepackage{mathtools}
\usepackage{xifthen}
\usepackage{todonotes}
\usepackage{enumitem}
\usepackage{stmaryrd}
\usepackage{constants}
\usepackage[ruled]{algorithm2e}
\usepackage{tikz,pgfplots}
\usepackage{chngcntr}
\usepackage{hyperref}
\usepackage[font={small}]{caption}

\counterwithin{figure}{section}

\graphicspath{{Pics/}}

\usepackage{constants}

\newconstantfamily{ctrlcst}{
	symbol=\kappa
}
\newconstantfamily{csts}{
	symbol=C
}
\renewconstantfamily{normal}{
	symbol=c
}

\makeatletter
\newcommand{\pushright}[1]{\ifmeasuring@#1\else\omit\hfill$\displaystyle#1$\fi\ignorespaces}
\newcommand{\pushleft}[1]{\ifmeasuring@#1\else\omit$\displaystyle#1$\hfill\fi\ignorespaces}
\makeatother

\newcommand{\bbZ}{\mathbb{Z}}
\newcommand{\bbR}{\mathbb{R}}

\newcommand{\Rd}{\mathbb{R}^d}
\newcommand{\Zd}{\mathbb{Z}^d}

\newcommand{\betac}{\beta_{\mathrm{\scriptscriptstyle c}}}
\renewcommand{\norm}[1]{\|#1\|}
\renewcommand{\abs}[1]{\lvert#1\rvert}
\newcommand{\normI}[1]{\left\|#1\right\|_{\scriptscriptstyle 1}}

\newcommand{\setof}[2]{\{#1\,:\,#2\}}

\newcommand{\given}{\,|\,}
\newcommand{\bgiven}{\bigm\vert}
\newcommand{\Bgiven}{\Bigm\vert}

\newcommand{\comp}{\mathrm{c}}
\newcommand{\eone}{\vec e_1}

\newcommand{\p}{\mathbb{P}}

\newcommand{\Z}{\mathbb{Z}}

\newcommand{\Ham}{\mathcal{H}}

\newcommand{\I}{\mathds{1}}

\renewcommand{\nleftrightarrow}{\mathrel{\ooalign{$\leftrightarrow$\cr\hidewidth$/$\hidewidth}}}

\renewcommand{\emptyset}{\varnothing}
\newcommand{\calA}{\mathcal{A}}
\newcommand{\calB}{\mathcal{B}}

\newcommand{\calG}{\mathcal{G}}
\newcommand{\normII}[1]{\|#1\|_{\scriptscriptstyle 2}}
\newcommand{\supp}{\mathrm{supp}}
\newcommand{\cov}{\mathrm{Cov}}
\newcommand{\uvec}{\mathbf{u}}
\newcommand{\bk}[1]{\langle #1 \rangle_\beta}
\newcommand{\bfJ}{\mathbf{J}}
\newcommand{\symmdiff}{\mathop{\triangle}}

\newcommand{\fcone}{\mathcal{Y}^\blacktriangleleft}
\newcommand{\bcone}{\mathcal{Y}^\blacktriangleright}

\newcommand{\diam}{D}

\newcommand{\vect}[1]{\overset{{}_{\shortrightarrow}}{#1}}
\newcommand{\tcev}[1]{\overset{{}_{\shortleftarrow}}{#1}}
\newcommand{\rhoL}{\rho_{\scriptscriptstyle\rm L}}
\newcommand{\rhoR}{\rho_{\scriptscriptstyle\rm R}}
\newcommand{\calBf}{\mathcal{B}^\blacktriangleleft}
\newcommand{\calBb}{\mathcal{B}^\blacktriangleright}

\newcommand{\walk}{\mathbf{S}}
\newcommand{\walkLaw}[1]{\mathbb{P}_{#1}}
\newcommand{\syncWalk}{\tilde{\walk}}
\newcommand{\syncWalkLaw}[1]{\tilde{\mathbb{P}}_{#1}}
\newcommand{\diffSyncWalk}{\check{\walk}}
\newcommand{\diffSyncWalkLaw}{\check{\mathbb{P}}}
\newcommand{\diffSyncWalkExp}{\check{\mathbb{E}}}
\newcommand{\hatSyncWalk}{\hat{\walk}}

\newcommand{\Hrange}{[\check{\walk}]^\parallel}

\newcommand{\isingLaw}[2]{\ifthenelse{\equal{#2}{}}{\langle#1\rangle_{\beta}}{\langle#1\rangle_{\beta,#2}}}
\newcommand{\isingPF}[1]{Z_{\beta,#1}}
\newcommand{\isingSet}[1]{\Omega_{#1}}

\newcommand{\evGraphSet}[2]{\ifthenelse{\equal{#2}{}}{\mathcal{E}_{#1}}{\mathcal{E}_{#1}(#2)}}
\newcommand{\evGraph}{\mathfrak{g}}
\newcommand{\evGraphLaw}[2]{\ifthenelse{\equal{#2}{}}{\mathfrak{P}^{#1}}{\mathfrak{P}^{#1}_{#2}}}

\newcommand{\randPathLaw}{\overline{\mathfrak{P}}}
\newcommand{\HTPF}[2]{\ifthenelse{\equal{#2}{}}{\mathbf{Z}^{\rm\scriptscriptstyle HT}(#1)}{\mathbf{Z}^{\rm\scriptscriptstyle HT}_{#2}(#1)}}

\newcommand{\currentSet}[2]{\ifthenelse{\equal{#2}{}}{\mathcal{N}_{#1}}{\mathcal{N}_{#1}(#2)}}
\newcommand{\current}{\mathbf{n}}
\newcommand{\currentw}[1]{\textnormal{w}_{\beta}(#1)}
\newcommand{\currentLaw}[2]{\ifthenelse{\equal{#2}{}}{\mathbf{P}^{#1}}{\mathbf{P}^{#1}_{#2}}}
\newcommand{\currentPF}[2]{\ifthenelse{\equal{#2}{}}{\mathbf{Z}^{\rm\scriptscriptstyle RC}(#1)}{\mathbf{Z}_{#2}^{\rm\scriptscriptstyle RC}(#1)}}
\newcommand{\FKPF}[1]{\mathbf{Z}^{\rm\scriptscriptstyle FK}_{#1}}

\newcommand{\RCMLaw}[2]{\ifthenelse{\equal{#2}{}}{\p\left(#1\right)}{\p_{#2}\left(#1\right)}}
\newcommand{\open}{\textnormal{\xspace open}}
\newcommand{\close}{\textnormal{\xspace closed}}
\newcommand{\openSet}[1]{\mathcal{O}(#1)}
\newcommand{\closeSet}[1]{\mathcal{C}(#1)}

\newcommand{\EvPart}[1]{\mathfrak{E}_{#1}}

\newcommand{\parityGraph}[2]{\ifthenelse{\equal{#2}{}}{\mathbf{G}_{#1}}{\mathbf{G}_{#1}(#2)}}

\newcommand{\frD}{\mathfrak{D}}
\newcommand{\calD}{\mathcal{D}}
\newcommand{\calP}{\mathcal{P}}
\newcommand{\calQ}{\mathcal{Q}}
\newcommand{\calR}{\mathcal{R}}

\newcommand{\nB}{{\vec{B}}}
\newcommand{\nA}{{\vphantom{\nB}A}{}}
\newcommand{\Path}[2][A]{\text{\sc Path}_G^{#1}(#2)}

\theoremstyle{plain}
\newtheorem{theorem}{Theorem}[section]
\newtheorem{lemma}[theorem]{Lemma}
\newtheorem{proposition}[theorem]{Proposition}

\newtheorem{remark}{Remark}[section]

\theoremstyle{definition}
\newtheorem{obs}{Observation}
\newtheorem{claim}{Claim}

\author{S\'{e}bastien Ott}
\address{Section de Mathématiques, Université de Genève, CH-1211 Genève, Switzerland}
\email{sebastien.ott@unige.ch}
\author{Yvan Velenik}
\address{Section de Mathématiques, Université de Genève, CH-1211 Genève, Switzerland}
\email{yvan.velenik@unige.ch}

\title{Asymptotics of even-even correlations in the Ising model}

\begin{document}

\begin{abstract}
We consider finite-range ferromagnetic Ising models on \(\Zd\) in the regime \(\beta<\betac\). We analyze the behavior of the prefactor to the exponential decay of \(\cov(\sigma_A,\sigma_B)\), for arbitrary finite sets \(A\) and \(B\) of even cardinality, as the distance between \(A\) and \(B\) diverges.
\end{abstract}

\maketitle

%
%

\section{Introduction and results}
\label{sec:Intro}

We consider the Ising model on \(\Zd\) with formal Hamiltonian
\[
\Ham = -\sum_{\{i,j\}\subset\Zd} J_{j-i} \sigma_i \sigma_j,
\]
where, as usual, \(\sigma_i\) denotes the random variable  corresponding to the spin at \(i\in\Zd\). The coupling constants \(\bfJ=(J(x))_{x\in\Zd}\) are assumed to satisfy the following conditions \footnote{Let us emphasize that these conditions are actually stronger than needed. In particular, the irreducibility condition could be substantially weakened at the cost of (minor) additional technicalities. However, ferromagnetism and short-range interactions (that is, \(J_x\leq e^{-c\norm{x}}\) for some \(c>0\)) are real restrictions.}:
\begin{itemize}
\item ferromagnetism: \(J_x\geq 0\) for all \(x\in\Zd\);
\item reflection symmetry: \(J_x = J_{\tilde x}\) if \(x=(x_1,\dots,x_d), \tilde{x}=(\tilde{x}_1,\dots,\tilde{x}_d)\) with \(|x_i|=|\tilde{x}_i|\) for all \(1\leq i\leq d\);
\item finite-range: \(\exists R<\infty\) such that \(J_x=0\) whenever \(\normII{x}\geq R\);
\item irreducibility: \(J_x>0\) for all \(x\in\Zd\) with \(\norm{x}=1\).
\end{itemize}

Let \(\calG_\beta\) be the set of all infinite-volume Gibbs measures at inverse temperature \(\beta\) and \(\betac = \sup\setof{\beta}{|\calG_\beta| = 1}\). As is well known, \(\betac\in(0,\infty)\) for all \(d\geq 2\). (We refer to~\cite{Friedli+Velenik-2017} for an introduction to these topics.)

\medskip
From now on, we suppose that \(d\geq 2\) and \(\beta<\betac\) and we denote by \(\mu_\beta\) the unique infinite-volume Gibbs measure. It was proved in~\cite{Aizenman+Barsky+Fernandez-1987} that, in this regime, spin-spin correlations decay exponentially fast: the inverse correlation length
\[
\xi_{\beta}(\uvec) = -\lim_{n\to\infty} \frac1n \log\cov_\beta(\sigma_0,\sigma_{[n\uvec]})
\]
exists and is positive, uniformly in the unit vector \(\uvec\) in \(\Rd\). Here, the covariance is computed with respect to \(\mu_\beta\), and \([y]\in\Zd\) denotes the component-wise integer part of \(y\in\Rd\).

\medskip
Given a function \(f:\{-1,1\}^{\Zd}\to\bbR\), we denote by \(\supp(f)\) the support of \(f\), that is, the smallest set \(A\) such that \(f(\omega)=f(\omega')\) whenever \(\omega\) and \(\omega'\) coincide on \(A\). \(f\) is said to be local if \(|\supp(f)|<\infty\).

Let \(\theta_x\) denote the translation by \(x\in\Zd\). \(\theta_x\) acts on a spin configuration \(\omega\) via \((\theta_x\omega)_y = \omega_{y-x}\) and on a function \(f\) via \(\theta_x f = f\circ \theta_{-x}\).

\medskip
We are interested in the asymptotic behavior of the covariances
\begin{equation}\label{eq:Covfg}
\cov_\beta(f, \theta_{[n\uvec]} g)
\end{equation}
as \(n\to\infty\), for pairs of local functions \(f\) and \(g\) and any unit vector \(\uvec\) in \(\Rd\). For any local function \(f\), there exist (explicit) coefficients \((\hat{f}_A)_{A\subset\supp(f)}\) such that
\[
f = \sum_{A\subset\supp(f)} \hat{f}_A \sigma_A ,
\]
where \(\sigma_A = \prod_{i\in A} \sigma_i\) (see, for example, \cite[Lemma~3.19]{Friedli+Velenik-2017}). Using this, \eqref{eq:Covfg} becomes
\begin{equation}\label{eq:expCov}
\cov_\beta(f, \theta_{[n\uvec]} g) = \sum_{\substack{A\subset\supp(f)\\B\subset\supp(g)}} \hat{f}_A \hat{g}_B \, \cov_\beta(\sigma_A, \sigma_{B+[n\uvec]}).
\end{equation}

This motivates the asymptotic analysis of the covariances \(\cov_\beta(\sigma_A, \sigma_{B+[n\uvec]})\) for \(A,B\Subset\Zd\) (that is, \(A\) and \(B\) are finite subsets of \(\Zd\)). Observe that, by symmetry, \(\cov_\beta(\sigma_A, \sigma_{B+[n\uvec]}) = 0\) whenever \(|A|\) and \(|B|\) have different parities. We can therefore assume, without loss of generality, that \(|A|\) and \(|B|\) are either both odd, or both even: one then says that one considers odd-odd, resp.\ even-even, correlations.

\subsection*{Odd-odd correlations}
In this case, the best nonperturbative result to date is the following:
\begin{theorem}[\cite{Campanino+Ioffe+Velenik-2004}]\label{thm:OZ-odd-odd}
Let \(A,B\Subset\Zd\) be two sets of odd cardinality. Then, for any \(d\geq 2\), any \(\beta<\betac\) and any unit vector \(\uvec\) in \(\Rd\), there exists a constant \(0 < C < \infty\) (depending on \(A,B,\uvec,\beta, \bfJ\)) such that
\[
\cov_\beta(\sigma_A,\sigma_{B+[n\uvec]}) = \frac{C}{n^{(d-1)/2}} e^{-\xi_{\beta}(\uvec) n} (1+o(1)) ,
\]
as \(n\to\infty\).
\end{theorem}
This result has a long history, starting with the celebrated work by Ornstein and Zernike in 1914~\cite{Ornstein+Zernike-1914,Zernike-1916}, in which the corresponding claim (for fluids) is established (non-rigorously) in the case of the 2-point function (that is, when \(A\) and \(B\) are singletons).
That such a behavior indeed occurs was first shown when \(d=2\) using exact computations~\cite{Wu-1966}.
In general dimensions, the earliest rigorous derivations, valid for sufficiently small \(\beta\), are due to Abraham and Kunz~\cite{Abraham+Kunz-1977} and Paes-Leme~\cite{Paes-Leme-1978}, using very different approaches. A non-perturbative derivation, valid for arbitrary \(\beta<\betac\) was given by Campanino, Ioffe and Velenik~\cite{Campanino+Ioffe+Velenik-2003} (see also~\cite{Campanino+Ioffe+Velenik-2008}).

Derivations for general odd-odd correlations were first obtained by Bricmont and Fröhlich~\cite{Bricmont+Frohlich-1985a,Bricmont+Frohlich-1985b} and Minlos and Zhizhina~\cite{Zhizhina+Minlos-1988,Minlos+Zhizhina-1996}, for sufficiently small values of \(\beta\), and then by Campanino, Ioffe and Velenik~\cite{Campanino+Ioffe+Velenik-2004} for all \(\beta<\betac\).

\subsection*{Even-even correlations}
The asymptotic behavior of even-even correlations is more subtle and initially led to some controversy: in the case of energy-energy correlations (that is, when \(A\) and \(B\) are each reduced to two nearest-neighbor vertices), heuristic derivations by Polyakov~\cite{Polyakov-1969} and Camp and Fisher~\cite{Camp+Fisher-1971} led to distinct predictions. While both works agreed that the rate of exponential decay is now given by \(2\xi_{\beta}(\uvec)\), the predicted behavior for the prefactor were different: Camp and Fisher predicted a prefactor of order \(n^{-d}\) for all \(d\geq 2\), while Polyakov predicted a prefactor of order \(n^{-2}\) when \(d=2\), \((n\log n)^{-2}\) when \(d=3\) and \(n^{-(d-1)}\) when \(d\geq 4\). When \(d=2\), these predictions coincided and matched the result known from exact computations~\cite{Stephenson-1966, Hecht-1967}. It turned out that Polyakov's predictions were correct. For sufficiently small values of \(\beta\), this was first proved by Bricmont and Fröhlich~\cite{Bricmont+Frohlich-1985a, Bricmont+Frohlich-1985b} for dimensions \(d\geq 4\), and then by Minlos and Zhizhina~\cite{Zhizhina+Minlos-1988, Minlos+Zhizhina-1996} for dimensions \(2\) and \(3\); see also~\cite{Auil-2002, Auil+Barata-2005}.
Extensions to general even-even correlations, at small values of \(\beta\), were proved in~\cite{Zhizhina+Minlos-1988, Minlos+Zhizhina-1996, Boldrighini+Minlos+Pellegrinotti-2011}.

\medskip
Our main result is  the following non-perturbative derivation.
\begin{theorem}\label{thm:OZ-even-even}
Let \(A,B\Subset\Zd\) be two nonempty sets of even cardinality. Then, for any \(d\geq 2\), any \(\beta<\betac\) and any unit vector \(\uvec\) in \(\Rd\), there exist constants \(0 < C_- \leq C_+ < \infty\) (depending on \(A,B,\uvec,\beta,\bfJ\)) such that, for all \(n\) large enough,
\[
\frac{C_-}{\Xi(n)} e^{-2\xi_{\beta}(\uvec) n}
\leq
\cov_\beta(\sigma_A,\sigma_{B+[n\uvec]})
\leq
\frac{C_+}{\Xi(n)} e^{-2\xi_{\beta}(\uvec) n} ,
\]
with
\[
\Xi(n) =
\begin{cases}
n^2				&	\text{when } d=2,\\
(n\log n)^2		&	\text{when } d=3,\\
n^{d-1}			&	\text{when } d\geq 4.
\end{cases}
\]
\end{theorem}
\begin{remark}
In addition to its being non-perturbative, an advantage of our approach is that it applies to arbitrary directions \(\uvec\) and a very general class of finite-range interactions, while the earlier approaches imposed severe limitations.
\end{remark}

\subsection*{Heuristics and scheme of the proof}
Let us try to provide some intuition for the behavior stated in Theorem~\ref{thm:OZ-even-even}, considering the simplest non-trivial case: \(A=\{x,y\}\), \(B+n\uvec=\{u,v\}\). The high-temperature representation of correlation functions (see Section~\ref{sec:HTrepr}) expresses the correlation function \(\bk{\sigma_A\sigma_{B+n\uvec}}\) as a sum over pairs of disjoint paths \(\gamma_1,\gamma_2\) with endpoints in \(\{x,y,u,v\}\), with a suitable weight associated to each realization of the two paths. The three possible pairing are thus:
\begin{center}
\resizebox{\textwidth}{!}{\begin{picture}(0,0)%
\includegraphics{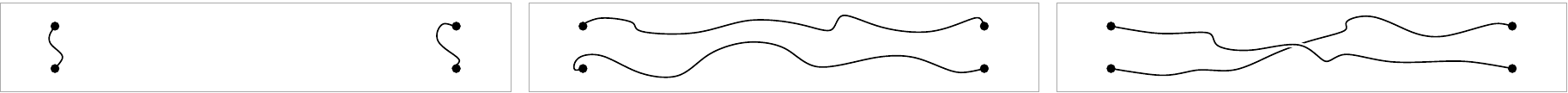}%
\end{picture}%
\setlength{\unitlength}{1036sp}%
\begingroup\makeatletter\ifx\SetFigFont\undefined%
\gdef\SetFigFont#1#2#3#4#5{%
  \reset@font\fontsize{#1}{#2pt}%
  \fontfamily{#3}\fontseries{#4}\fontshape{#5}%
  \selectfont}%
\fi\endgroup%
\begin{picture}(33391,1921)(630,-2375)
\put(33166,-1231){\makebox(0,0)[lb]{\smash{{\SetFigFont{12}{14.4}{\rmdefault}{\mddefault}{\updefault}{\color[rgb]{0,0,0}$u$}%
}}}}
\put(33166,-2131){\makebox(0,0)[lb]{\smash{{\SetFigFont{12}{14.4}{\rmdefault}{\mddefault}{\updefault}{\color[rgb]{0,0,0}$v$}%
}}}}
\put(23536,-1231){\makebox(0,0)[lb]{\smash{{\SetFigFont{12}{14.4}{\rmdefault}{\mddefault}{\updefault}{\color[rgb]{0,0,0}$x$}%
}}}}
\put(23536,-2131){\makebox(0,0)[lb]{\smash{{\SetFigFont{12}{14.4}{\rmdefault}{\mddefault}{\updefault}{\color[rgb]{0,0,0}$y$}%
}}}}
\put(21916,-1231){\makebox(0,0)[lb]{\smash{{\SetFigFont{12}{14.4}{\rmdefault}{\mddefault}{\updefault}{\color[rgb]{0,0,0}$u$}%
}}}}
\put(21916,-2131){\makebox(0,0)[lb]{\smash{{\SetFigFont{12}{14.4}{\rmdefault}{\mddefault}{\updefault}{\color[rgb]{0,0,0}$v$}%
}}}}
\put(12286,-1231){\makebox(0,0)[lb]{\smash{{\SetFigFont{12}{14.4}{\rmdefault}{\mddefault}{\updefault}{\color[rgb]{0,0,0}$x$}%
}}}}
\put(12286,-2131){\makebox(0,0)[lb]{\smash{{\SetFigFont{12}{14.4}{\rmdefault}{\mddefault}{\updefault}{\color[rgb]{0,0,0}$y$}%
}}}}
\put(10666,-1231){\makebox(0,0)[lb]{\smash{{\SetFigFont{12}{14.4}{\rmdefault}{\mddefault}{\updefault}{\color[rgb]{0,0,0}$u$}%
}}}}
\put(10666,-2131){\makebox(0,0)[lb]{\smash{{\SetFigFont{12}{14.4}{\rmdefault}{\mddefault}{\updefault}{\color[rgb]{0,0,0}$v$}%
}}}}
\put(1036,-1231){\makebox(0,0)[lb]{\smash{{\SetFigFont{12}{14.4}{\rmdefault}{\mddefault}{\updefault}{\color[rgb]{0,0,0}$x$}%
}}}}
\put(1036,-2131){\makebox(0,0)[lb]{\smash{{\SetFigFont{12}{14.4}{\rmdefault}{\mddefault}{\updefault}{\color[rgb]{0,0,0}$y$}%
}}}}
\put(2116,-1681){\makebox(0,0)[lb]{\smash{{\SetFigFont{12}{14.4}{\rmdefault}{\mddefault}{\updefault}{\color[rgb]{0,0,0}$\gamma_1$}%
}}}}
\put(9271,-1681){\makebox(0,0)[lb]{\smash{{\SetFigFont{12}{14.4}{\rmdefault}{\mddefault}{\updefault}{\color[rgb]{0,0,0}$\gamma_2$}%
}}}}
\end{picture}%
}
\end{center}
For the product \(\bk{\sigma_A}\bk{\sigma_{B+n\uvec}}\), one obtains a representation similar to the left-most one above, but with the two paths \(\gamma_1,\gamma_2\) now living in independent copies of the system (and not necessarily disjoint). Since \(\beta<\betac\), typical paths between \(x\) and \(y\) and between \(u\) and \(v\) do not wander far from their endpoints. Therefore, when \(n\gg 1\), one might expect that the contribution from the left-most of the above pictures essentially factorizes, thus canceling the contribution of \(\bk{\sigma_A}\bk{\sigma_{B+n\uvec}}\) and leaving only the contributions of the two right-most pictures. If the two paths could be considered independent, one would then recover the square of the usual Ornstein--Zernike decay. There are however, several problems with this argument (which, in particular, are responsible for the deviations from this behavior in dimensions \(2\) and \(3\)). First, it turns out that the interaction between the two paths in the left-most picture, although small, decays exponentially in \(n\) with a rate \(2\xi_{\beta}(\uvec)\), that is, the correction is of the same (exponential) order as the target estimates in Theorem~\ref{thm:OZ-even-even}. Therefore, the behavior of the truncated 4-point function cannot be read solely from the last two pictures. Second, there is a non-trivial infinite-range interaction between the two paths (as well as self-interactions) that make approximating them by independent random walks delicate. Third, the constraint that the two paths do not intersect, although mostly irrelevant in dimensions \(d\geq 4\), is crucial when \(d=2\) or \(3\) and is ultimately responsible for the anomalous behavior observed in these dimensions.

In order to solve these problems, we rely on a combination of several graphical representations of the Ising model.
Name\-ly, it is well known that the random-current representation offers an extremely efficient way of dealing with truncated correlations, expressing the difference as the probability of a suitable event in a duplicated system. In particular, it suppresses the need of separately estimating the 4-point function and the product of the 2-point functions.
Using this and a coupling of the resulting double random-current configurations with the paths from the corresponding high-temperature representation, we obtain an upper bound on the truncated correlation function in terms two high-temperature paths connecting vertices in \(A\) to vertices in \(B\) (that is, roughly speaking, to the two right-most pictures above). The remarkable feature is that these paths live in two \emph{independent} copies of the system and are only coupled through the constraint that they do not intersect (see Fig.~\ref{fig:heuristics}).

This approach does not apply for the lower bound. In this case, we work directly in the random-cluster representation, observing that the FKG inequality allows one to cancel (as a lower bound) the product of the 2-point functions with a suitable part of the 4-point function. Again the resulting picture is that of two clusters connecting vertices in \(A\) to vertices in \(B\), living in independent copies of the system and coupled through a non-intersection constraint.

Using the Ornstein--Zernike (OZ) theory developed in~\cite{Campanino+Ioffe+Velenik-2003,Campanino+Ioffe+Velenik-2008,Ott+Velenik-2017}, we then approximate the high-temperature paths, resp.\ the clusters in the random-cluster representation, by effective directed random walks on \(\Zd\). This allows one to obtain upper and lower bounds given by the square of the usual OZ asymptotics multiplied by the probability that the two effective random walk bridges do not intersect. The latter being \(\Theta(n^{-1})\) when \(d=2\), \(\Theta((\log n)^{-2})\) when \(d=3\) and \(\Theta(1)\) when \(d\geq 4\), this leads to the claim of Theorem~\ref{thm:OZ-even-even}.

\subsection*{Open problems}
\subsubsection*{Sharp asymptotics.}
In contrast to Theorem~\ref{thm:OZ-odd-odd}, Theorem~\ref{thm:OZ-even-even} only provides bounds, not sharp asymptotics. It would be desirable to remove this limitation (note that such sharp asymptotics were obtained in the earlier approaches, albeit only for \(\beta\) small enough).
There are several places in which we lose track of the sharp prefactor, but only two steps where this is essential (the inequalities in~\eqref{ineq:crucial1} and~\eqref{ineq:crucial2}), all others could be dealt with at the cost of (non-negligible) additional technicalities.
One way to obtain sharp asymptotics may be to build a version of the OZ theory applicable directly in the (double) random-current setting, maybe by building on the construction in~\cite{Ott2018}.

\subsubsection*{Short-range interactions.}
We only consider finite-range interactions. However, the same results should hold for infinite-range interactions, as long as those decay at least exponentially fast with the distance. Again, this would require a suitable extension of the OZ theory. We plan to come back to this issue in a future work.

\subsubsection*{Rate of exponential decay.}
The random-current representation plays an essential role in our proof. Even proving that the rate of exponential decay is \(2\xi(\uvec)\) does not seem to be immediate using only the random-cluster representation. This would be necessary to investigate similar questions in the Potts model.

\subsubsection*{Other models.}
The first thing to understand in the context of more general models (in particular, with a richer symmetry group) is to determine what are the relevant classes of functions (playing the role of the \(\sigma_A\), with \(|A|\) even or odd, in the Ising model). Of course, characters of the symmetry group seem the natural generalization, but in which classes should they be split?

\begin{figure}
\centering
\includegraphics[height=3cm]{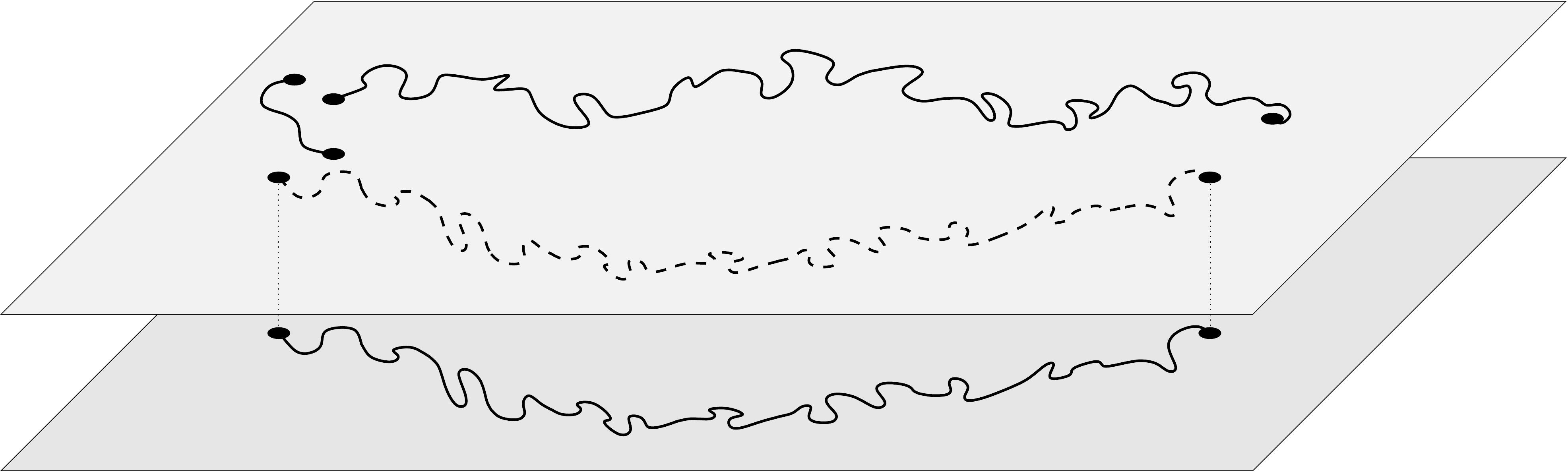}\\[5mm]
\includegraphics[height=3cm]{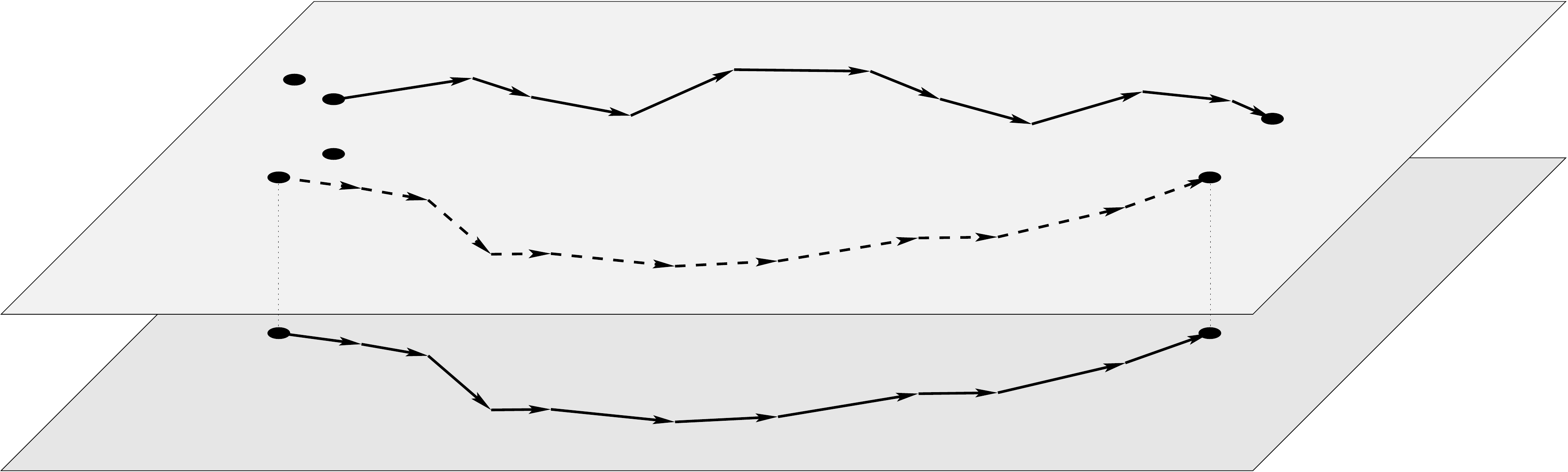}
\caption{\emph{Top:} A sketch of the graphical representation corresponding to the upper bound on the truncated correlation function \(\cov_\beta(\sigma_A,\sigma_B)\) (here with \(|A|=4\) and \(|B|=2\)). There are two paths, living in independent copies of the system, each connecting one vertex of \(A\) to one vertex of \(B\) and non-intersecting. The remaining two vertices of \(A\) are connected to each other. \emph{Bottom:} the corresponding non-intersecting directed random walks.}
\label{fig:heuristics}
\end{figure}

%
%

\section*{Acknowledgments}
The authors gratefully acknowledge the support of the Swiss National Science
Foundation through the NCCR SwissMAP.

%
%

\section{Proof of Theorem~\ref{thm:OZ-even-even}}

\subsection{The graphical representations}

We consider the ferromagnetic Ising model on a finite graph $G=(V_G,E_G)$ with free boundary conditions (since $\beta<\betac$ in our application, which boundary condition is used does not matter). The set of configurations is $\isingSet{G} = \{-1,+1\}^{V_G}$, and the expectation of a function \(f\) is given by
\begin{equation}
	\label{eq:IsingMeasure}
	\isingLaw{f}{G} = \frac{1}{\isingPF{G}} \sum_{\sigma\in\isingSet{G}} f(\sigma) \exp\Bigl(\beta\sum_{e=\{i,j\}\in E_G} J_e\sigma_i\sigma_j\Bigr),
\end{equation}
where the coupling constants \((J_e)_{e\in E_G}\) are nonnegative.

\smallskip
We now describe three graphical representations of the Ising model: the random-current (RC) representation with its switching property, the random-cluster (FK) representation and the high-temperature (HT) representation.
We will make the convention of systematically identifying a sub-graph $F\subset G$ with the induced function on edges equal to $1$ when the edge is present in $F$ and to $0$ otherwise.
General references for this section are~\cite{Friedli+Velenik-2017, Duminil-Copin-2017}.

\subsubsection{Random-current representation}
The random-current measure on $G$ is the non-negative measure on $(\Z_{\geq 0})^{E_G}$ associating to \(\current=(\current_e)_{e\in E_G}\) the weight
\[
\currentw{\current} = \prod_{e\in E_G} \frac{(\beta J_e)^{\current_e}}{\current_e !}.
\]
Given a random-current configuration $\current$, the incidence of a vertex $v\in V_N$ is defined as
\[
I_v(\current) = \sum_{e\in E_G, e\ni v}\current_e.
\]
The sources $\partial\current$ of a configuration \(\current\) are the vertices with odd incidence. Let us now express the unnormalized correlation functions of the Ising model with free boundary condition in terms of random currents. Let \(A\subset V_G\) with \(|A|\) even.
A Taylor expansion of the Boltzmann weight yields
\begin{align*}
	\sum_{\sigma\in\Omega_{G}}\sigma_{A}\prod_{e=\{i,j\}\in E_G} e^{\beta J_e\sigma_i\sigma_j} &= \Bigl(\prod_{i\in V_G}\sum_{\sigma_i\in\{-1,1\}}\Bigr)\sigma_{A}\sum_{\current:E_G\to \Z_{\geq 0}} \currentw{\current}  \prod_{i\in V_G}(\sigma_i)^{I_i(\current)}\\
	&= 2^{|V_G|}\sum_{\partial\current=A} \currentw{\current}.
\end{align*}
We will write
\begin{align*}
	\currentPF{A}{G} &= \sum_{\partial\current=A} \currentw{\current},\\
	\currentPF{A}{G}\currentPF{B}{G}\left\{F(\current_1+\current_2)\right\} &= \sum_{\substack{\partial\current_1=A\\ \partial\current_2=B}} \currentw{\current_1}\currentw{\current_2} F(\current_1+\current_2),
\end{align*}
and abuse notation by replacing the indicator function of an event by the event inside the braces. The set of currents on $G$ with sources $A$ will be denoted $\currentSet{G}{A}$ (with $\currentSet{G}{}\equiv\currentSet{G}{\emptyset}$). Correlation functions of the Ising model then become
\[
\isingLaw{\sigma_A}{G} = \frac{\currentPF{A}{G}}{\currentPF{\emptyset}{G}}.
\]
For a current configuration $\current$, denote by $\widehat{\current}$ the graph obtained by keeping the edges from $G$ with $\current_e>0$ and removing those with $\current_e=0$. The feature that makes this representation extremely useful is the following \emph{Switching Lemma} (see \cite{Duminil-Copin-2017} for a proof and applications).
\begin{lemma}[Switching Lemma]\label{lem:switching}
	For any $A,B\subset V_G$ and any function $F$ on currents,
	\[
	\currentPF{A}{G} \currentPF{B}{G} \bigl\{F(\current_1+\current_2)\bigr\}
	=
	\currentPF{A\symmdiff B}{G} \currentPF{\emptyset}{G} \bigl\{\mathds{1}_{\mathfrak{E}_A}(\widehat{\current_1+\current_2}) F(\current_1+\current_2) \bigr\},
	\]
where $\mathfrak{E}_A$ is the event that $A$ is evenly partitioned by $\widehat{\current_1+\current_2}$ (that is, each connected component of \(\widehat{\current_1+\current_2}\) contains an even number of vertices of \(A\)) and \(A\symmdiff B = (A\cup B) \setminus (A\cap B)\).
\end{lemma}
This remarkable property makes the random-current representation particularly well suited to analyze truncated correlation functions. In particular, an application of Lemma~\ref{lem:switching} yields, for any \(A,B\subset G\),
\begin{equation}\label{eq:CovABRC}
	\isingLaw{\sigma_A;\sigma_B}{}
	=
	\frac{\currentPF{A\symmdiff B}{} \currentPF{\emptyset}{} - \currentPF{A}{}\currentPF{B}{}}{\currentPF{\emptyset}{}^2}
	=
	\frac{\currentPF{A\symmdiff B}{} \currentPF{\emptyset}{} \{\EvPart{A}^{\,\comp}\}}{\currentPF{\emptyset}{}^2}.
\end{equation}

\smallskip
Finally, we introduce a probability measure on $\currentSet{G}{A}$ via
\[
\currentLaw{A}{G}(\current)
=
\frac	{ \currentw{\current} }
		{ \currentPF{A}{G} } ,
\]
and write, for \(A,B\subset V_G\),
\[
\currentLaw{A\otimes B}{G}(\current_1,\current_2)
=
\currentLaw{A}{G}(\current_1) \currentLaw{B}{G}(\current_2) .
\]

\subsubsection{High-temperature representation}
\label{sec:HTrepr}

For any $A\subset V_G$ with \(|A|\) even, using the identity $e^{\beta J_e\sigma_i\sigma_j} = \cosh(\beta J_e) \bigl(1+\sigma_i\sigma_j\tanh(\beta J_e)\bigr)$, we can write
\begin{align*}
	\sum_{\sigma\in\isingSet{G}}& \sigma_A \prod_{e=\{i,j\}\in E_G} e^{\beta J_e \sigma_i\sigma_j}\\
	&=
	\Bigl\{\prod_{e=\{i,j\}\in E_G} \cosh(\beta J_e)\Bigr\} \sum_{\evGraph\subset E_G}\sum_{\sigma\in\isingSet{G}} \prod_{i\in V_G}(\sigma_i)^{\mathds{1}_A(i)+|\setof{j}{\{i,j\}\in\evGraph}|} \prod_{e\in\evGraph}\tanh(\beta J_e)\\
	&=
	2^{|V_G|} \Bigl\{\prod_{e=\{i,j\}\in E_G} \cosh(\beta J_e)\Bigr\} \sum_{\evGraph\in\evGraphSet{G}{A}} \prod_{e\in\evGraph} \tanh(\beta J_e),
\end{align*}
where $\evGraphSet{G}{A}$ is the set of subgraphs $(V_G,E)$ of $G$ having the property that every vertex not in $A$ has even degree while every vertex in $A$ has odd degree. Correlation functions can then be expressed as
\[
\isingLaw{\sigma_A}{G}
=
\frac
	{\HTPF{A}{G}}
	{\HTPF{\emptyset}{G}},
\]
where we introduced
\[
	\HTPF{A}{G} = \sum_{\evGraph\in\evGraphSet{G}{A}} \prod_{e\in\evGraph} \tanh(\beta J_e) .
\]
We define an associated probability measure $\evGraphLaw{A}{G}$ on $\evGraphSet{G}{A}$ by
\[
	\evGraphLaw{A}{G}(\evGraph)
	=
	\frac
		{\prod_{e\in\evGraph} \tanh(\beta J_e)}
		{\HTPF{A}{G}} .
\]

\begin{remark}\label{rem:couplingRCHT}
Note that it immediately follows from the definitions that
\[
	\currentLaw{A}{G}(\current_e \text{ is odd if and only if }e\in\evGraph) = \evGraphLaw{A}{G}(\evGraph)
\]
for all subgraphs \(\evGraph\subset G\).
\end{remark}

Given \(u\in V_G\), we denote by \(E_G^u\) the set of all edges of \(E_G\) incident at \(u\). At each vertex \(u\), we fix an arbitrary total ordering on the edges of \(E_G^u\).

Given a vertex \(x\in V_G\) and a configuration \(\evGraph\in\evGraphSet{G}{A}\), we extract a path \(\gamma\) starting from \(x\) and ending at another vertex of \(A\) using the following algorithm (writing $|A|=m$):

\begin{algorithm}[H]
\DontPrintSemicolon
\label{alg:HTpaths}
	let $\gamma=\emptyset$ and $v = x$ \;
	\Repeat{$v\in A\setminus\{x\}$}{
		let $e=\{v,u\}$ be the smallest edge in $\evGraph\cap E_G^v$ that does not belong to \(\gamma\)\;
		update \(\gamma = \gamma \sqcup e\) \tcp*{\(\sqcup\) denotes concatenation}
		update $v=u$\;
	}
\caption{Extracting a path \(\gamma\) starting at \(x\in A\) from \(\evGraph\)}
\end{algorithm}

\begin{figure}
\resizebox{10cm}{!}{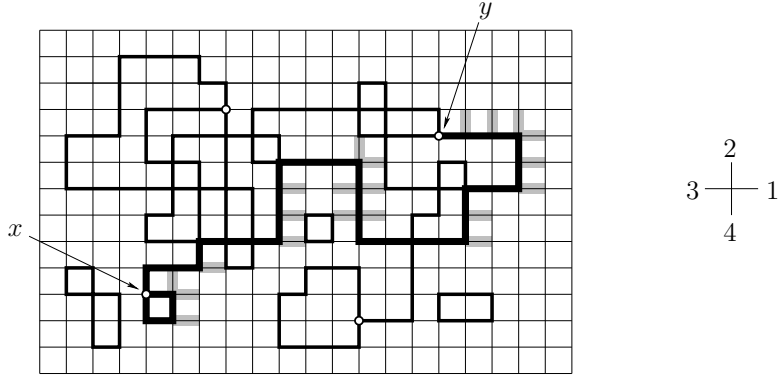}
\caption{A high-temperature configuration, with \(4\) vertices of odd degree. The path from \(x\) to \(y\) obtained by applying Algorithm~\ref{alg:HTpaths} is indicated by bold edges; the edges that need to be added to obtain the corresponding edge-boundary are shaded. The order used on the edges incident at each vertex is indicated on the right.}
\label{fig:HTpath}
\end{figure}

\smallskip
If \(\gamma=(x_0,x_1,\ldots,x_m)\), we define its edge-boundary by
\[
	[\gamma] = \bigcup_{k=0}^{m-1} \setof{e\in E_G^{x_k}}{e\leq e_k} ,
\]
where we have written \(e_k=\{x_{k},x_{k+1}\}\) for the edges of the path.

Given \(x,y\in A\), let us denote by \(\Path{x,y}\) the set of all paths in \(G\) starting at \(x\), ending at \(y\) that can be obtained using Algorithm~\ref{alg:HTpaths} above.
Given \(\gamma\in\Path{x,y}\), we denote by \(\{\gamma: x\rightarrow y\}\) the event that Algorithm~\ref{alg:HTpaths} (started at \(x\)) yields the path \(\gamma\) and by \(G[\gamma]\) the graph with edges \(E_G\setminus [\gamma]\) (and vertices given by the endpoints of these edges).
Then, one can easily check that
\[
	\evGraphLaw{A}{G}(\gamma: x\rightarrow y)
	=
	\frac{\HTPF{A\setminus\{x,y\}}{G[\gamma]}}{\HTPF{A}{G}}\, \prod_{e\in\gamma} \tanh(\beta J_e).
\]

The following upper bound will be useful in our analysis.
\begin{lemma}\label{lem:HTPathMonotonicity}
Let \(A\subset V_G\) be even. Let \(x,y\) be two distinct vertices of \(A\) and let \(\gamma\in\Path{x,y}\). Then,
\[
	\evGraphLaw{A}{G}(\gamma: x\rightarrow y)
	\leq
	\evGraphLaw{\{x,y\}}{G}(\gamma: x\rightarrow y) .
\]
\end{lemma}
\begin{proof}
Let us write
\[
	\frac{\HTPF{A\setminus\{x,y\}}{G[\gamma]}}{\HTPF{A}{G}}
	=
	\frac{\HTPF{A\setminus\{x,y\}}{G[\gamma]}}{\HTPF{\emptyset}{G[\gamma]}} \,
	\frac{\HTPF{\emptyset}{G[\gamma]}}{\HTPF{\emptyset}{G}} \,
	\frac{\HTPF{\emptyset}{G}}{\HTPF{A}{G}}	.
\]
Now, observe first that, by GKS inequalities,
\[
	\frac{\HTPF{A\setminus\{x,y\}}{G[\gamma]}}{\HTPF{\emptyset}{G[\gamma]}}
	=
	\isingLaw{\sigma_{A\setminus\{x,y\}}}{G[\gamma]}
	\leq
	\isingLaw{\sigma_{A\setminus\{x,y\}}}{G}.
\]
Then, again by GKS,
\[
	\frac{\HTPF{A}{G}}{\HTPF{\emptyset}{G}}
	=
	\isingLaw{\sigma_A}{G}
	\geq
	\isingLaw{\sigma_{A\setminus\{x,y\}}}{G} \,
	\isingLaw{\sigma_x \sigma_y}{G}	.
\]
Putting this together, we obtain
\[
	\frac{\HTPF{A\setminus\{x,y\}}{G[\gamma]}}{\HTPF{A}{G}}
	\leq
	\frac{1}{\isingLaw{\sigma_x \sigma_y}{G}}
	\frac{\HTPF{\emptyset}{G[\gamma]}}{\HTPF{\emptyset}{G}}
	=
	\frac{\HTPF{\emptyset}{G[\gamma]}}{\HTPF{\{x,y\}}{G}} .
\]
The claim follows, since \(\Path{x,y}\subset\Path[\{x,y\}]{x,y}\).
\end{proof}

We will denote by $\randPathLaw^{x,y}_G$ the distribution of the path extracted from HT configurations with sources $\{x,y\}$ (that is, the pushforward measure of $\evGraphLaw{\{x,y\}}{G}$ by the mapping induced by Algorithm~\ref{alg:HTpaths}) and by $\randPathLaw^{\{x,y\}\otimes \{u,v\}}_G$ the distribution of two independent such paths extracted from independent configurations with sources $\{x,y\}$ and $\{u,v\}$.

\subsubsection{Random-cluster representation}
The random-cluster (or FK) measure associated to the Ising model is obtained by the following expansion (remember that \(\EvPart{A}\) is the event that each connected component of \(\omega\) contains an even number of vertices of \(A\) (possibly \(0\))):
\begin{align*}
	\sum_{\sigma\in\{\pm 1\}^G} \sigma_A e^{\beta\sum_{e=\{i,j\}}J_e\sigma_i\sigma_j }
	&=
	e^{-\beta\sum_{e} J_e}\sum_{\sigma\in\{\pm 1\}^G} \sigma_A \prod_{e=\{i,j\}}\bigl( (e^{2\beta J_e}-1)\delta_{\sigma_i,\sigma_j} +1\bigr)\\
	&=
	C_{\beta,J}\sum_{\omega\subset E_G} 2^{\kappa(\omega)} \mathds{1}_{\EvPart{A}}(\omega) \prod_{e\in\omega}(e^{2\beta J_e}-1),
\end{align*}
where $\kappa_G(\omega)$ is the number of connected components of $(V_G,\omega)$ (each isolated vertex thus counting as a connected component).

Defining a probability measure \(\mathbb{P}_{\beta,G}\) on subsets \(\omega\subset E_G\) by
\[
\RCMLaw{\omega}{\beta,G}
=
\frac
	{2^{\kappa_G(\omega)} \prod_{e\in\omega}(e^{2\beta J_e}-1)}
	{\FKPF{G}},
\]
one obtains
\begin{equation}\label{eq:FKsigmaA}
\isingLaw{\sigma_A}{G}
=
\RCMLaw{\EvPart{A}}{\beta,G}.
\end{equation}

An event \(\calA\) is said to be increasing if \(\omega\in\calA\) implies that \(\omega'\in\calA\) for any \(\omega'\supset\omega\).
An important feature of this representation is that the FKG inequality holds~\cite{Duminil-Copin-2017,Friedli+Velenik-2017}: if \(\calA,\calB\) are two increasing events, then
\(
\RCMLaw{\calA\cap\calB}{\beta,G}
\geq
\RCMLaw{\calA}{\beta,G} \RCMLaw{\calB}{\beta,G} .
\)

%
%
\subsection{Notations, conventions}

In the sequel, we make the following conventions:
\begin{itemize}
\item \(c_1, c_2,\dots\) will denote generic (positive, finite) constants, whose value can change from place to place (even in consecutive lines);
\item \(\norm{\cdot}\) will denote the Euclidean norm;
\item to lighten notations, when a quantity should be an integer but the corresponding expression yields a non-integer, we will implicitly assume that the integer part has been taken;
\item we write occasionally \(\p(A,B)\) instead of \(\p(A\cap B)\);
\item if \(E\) is a set of edges and \(i\) a vertex, the notation \(i\in E\) will mean that at least one edge in \(E\) has \(i\) as an endpoint.
\end{itemize}
Moreover, we will always work in finite volume. More precisely, the expectations in Theorem~\ref{thm:OZ-even-even} will be computed under the finite-volume Gibbs measure with free boundary condition on the graph \(G=(V_G,E_G)\) with \(V_G= V_G(N) = \setof{x=(x_1,\dots,x_d)\in\Zd}{\max_i |x_i| \leq N}\) and \(E_G = E_G(N) = \setof{\{i,j\}\subset V_N}{J_{j-i}\neq 0}\).
We will assume that \(N\gg n\) (say, \(N=n^2\)); exponential decay of correlations then implies that all our estimates below are uniform in \(N\) and the thermodynamic limit \(N\uparrow\infty\) can be taken in the end.

Also, by symmetry, we can (and will) suppose, without loss of generality, that the unit vector \(\uvec=(u_1,\ldots,u_d)\in\Rd\) appearing in the statement of Theorem~\ref{thm:OZ-even-even} satisfies \(u_1 \geq \abs{u_k}\) for all \(k=2,\ldots,d\).

Finally, to lighten notation, we will write \(\nB = B + n\uvec\). The unit vector \(\uvec\) and the value of \(n\) will be kept fixed and are thus not indicated explicitly. \(n\) will be assumed to be large.

%
%

\subsection{Coupling with directed random walks}

\label{sec:RWrep}

In this section, we briefly summarize results from~\cite{Campanino+Ioffe+Velenik-2003, Campanino+Ioffe+Velenik-2008, Ott+Velenik-2017} that provide random-walk representations for both paths extracted from the HT expansion and for long subcritical clusters in the FK representation.

\subsubsection{The basic Ornstein-Zernike coupling}\label{sec:OZcoupling}
Fix \(x,x'\in\Zd\) and a unit-vector \(\uvec\in\Rd\), and set \(y=x'+[n\uvec]\).

Fix some \(\delta\in(0,1/\sqrt{2})\) and let \(\fcone = \setof{t\in\Rd}{\langle t, \eone\rangle \geq \delta\norm{t}}\) and \(\bcone=-\fcone\) be ``forward'' and ``backward'' cones with apex at \(0\), an aperture strictly larger than \(\pi/2\) (specified by \(\delta\)) and axis given by the first coordinate axis. Given \(v\in\Rd\), denote by \(\fcone_v=v+\fcone\), \(\bcone_v=v+\bcone\) and set \(\diam(v,w) = \fcone_v \cap\bcone_w\) for any \(w\in\fcone_v\).

Denote
\begin{gather*}
	\calA = \bigcup\limits_{v\in \fcone\cap\Zd} \{ (C,v):\ C\ni 0,v,\ C \text{ connected},\ C\subset\diam(0,v)\}\\
	\calBb = \bigcup\limits_{v\in \fcone\cap\Zd} \{ (C,v):\ C\ni 0,v,\ C \text{ connected},\ C\subset\bcone_v\}\\
	\calBf = \bigcup\limits_{v\in \fcone\cap\Zd} \{ (C,v):\ C\ni 0,v,\ C \text{ connected},\ C\subset\fcone\},
\end{gather*}
the \(C\) component is the ``cluster'' part while \(v\) is the \emph{displacement} of \(C\). Denote \( X=((B_L,v_L),(B_R,v_R),(C_1,v_1),(C_2,v_2),\dots) \) an element of \(\calBb\times \calBf\times\calA^{\Z_{\geq 1}}\). For such an \(X\), \(m\geq 0\) and \(x\in\Zd\), one can create a cluster connecting \(x\) to \(x+v_L+v_R+\sum_{k=1}^{m}v_i\) by looking at:
\[
C_x(m,X)=(x+B_L)\cup (x+v_L+C_1)\cup (x+v_L+v_1+ C_2)\cup\dots\cup (x+v_L+\sum_{k=1}^{m}v_i + B_R) .
\]
Then define the following event and function on \(\calBb\times \calBf\times\calA^{\Z_{\geq 1}}\):
\begin{gather*}
	R_x(y) = \{X:\ \exists m\geq 1: x+v_L+v_R+\sum_{k=1}^{m} v_k = y \}\\
	M_{y}(X) = \max\{m\geq 0:\sum_{k=1}^{m}(v_i)_1\leq y_1-(v_L)_1-(v_R)_1 \}.
\end{gather*}
Notice that \(M_{y-x}(X)\geq 1\) when \(X\in R_x(y) \).
Following~\cite{Campanino+Ioffe+Velenik-2003, Campanino+Ioffe+Velenik-2008} and~\cite[Section~4 and Appendix~C]{Ott+Velenik-2017}, one can construct  a probability measure \(p\) on \(\calA\) depending on \(u\) and two finite measures
\(\rhoL\) on \(\calBb\) and \(\rhoR\) on \(\calBf\) (depending on \(x,x'\) and \(\uvec\), recall \(y=\uvec n + x'\)\footnote{We emphasize that the constructions in~\cite{Campanino+Ioffe+Velenik-2003, Campanino+Ioffe+Velenik-2008, Ott+Velenik-2017} are actually explicit (in particular the measures and the coupling discussed here), but we only formulate the results in the form we need for the present work.}), for which the following holds:
\begin{enumerate}[label=P\arabic*]
	\item \label{it:expDecSteps} there exists \(\Cl{expDecSteps}>0\) such that:
	\[
	\rhoL(\norm{v}\geq k)\vee \rhoR(\norm{v}\geq k)\vee p(\norm{v}\geq k)\leq e^{-\Cr{expDecSteps}k},
	\]
	for \(k\) large enough; \(\Cr{expDecSteps}\) can be chosen to be uniform over \(x,y\).
	\item \label{it:expDecTVdist} Denoting \(\Xi^{y}_{x} \) the product measure \(\rhoL\times\rhoR\times p^{\Z_{\geq 1}}\) conditioned on \(R_x(y)\), there exists \(\Cl{expDecTVdist}>0\) such that:
	\[
	d_{TV}\bigl(\Xi^y_x(\ \cdot\ ), \RCMLaw{\ \cdot\ \given x \leftrightarrow y}{\beta} \bigr)\leq e^{-\Cr{expDecTVdist} n},
	\]
	provided that \(n\) be large enough, where \(\Xi^y_x(\ \cdot\ )\) is understood as a measure on the cluster \(C_x(M_{y-x}(X),X)\).
	\item \label{it:finEnergEone} \(p(v=\eone)>0\) and \(p(v=(2,z))>0\) for all \(z\in\Z^{d-1}\) with \(\normI{z}=1\).
	\item \label{it:finEnergBC} Let \((C_L,v_L)\in\calBb\). If there exists \(c=c(\beta,d)>0\) such that
	\[
		\RCMLaw{C_{xy}\cap \bcone_{x+v_L}=x+C_L\given x \leftrightarrow y}{\beta}\geq c,
	\]
	then there exists \(c'=c'(\beta,d)>0\) such that \[\rhoL\bigl((C,v)=(C_L,v_L)\bigr)\geq c'.\] The same holds for \((C_R,v_R)\in\calBf\) under \(\rhoR\).
\end{enumerate}

Relevant statements in~\cite{Ott+Velenik-2017} for those properties are: Theorem C.4 for Item~\ref{it:expDecSteps}, Lemma C.1 and Theorem C.4 for Item~\ref{it:expDecTVdist}, Lemma C.3 and Theorem C.4 for Item~\ref{it:finEnergEone} and Remark C.1 for Item~\ref{it:finEnergBC}.

We will denote \(C_{x,y}\equiv C_x(M_{y-x}(X),X)\).

\begin{remark}
	\label{rem:coneConfinement}
	It follows from the above properties that:
	\begin{itemize}
		\item Given \(v_L,v_R\), \(\Xi_{x}^y(\ \cdot\ \given v_L,v_R )\) is the law of a directed random walk with steps sampled according to \(p\), conditioned to go from \(x+v_L\) to \(y-v_R\).
		\item Given the displacements \((v_L,v_1,\dots,v_{M_{y-x}},v_R)\), the cluster \(C_x(M_{y-x}(X),X)\) obtained from \(\Xi_x^y\) is contained in the diamond cover \(\bcone_{x+v_L}\cup\diam(x+v_{L},x+v_L+v_1)\cup\dots\cup\fcone_{y-v_R} \) (see Figure~\ref{fig:EffRW}).
	\end{itemize}
\end{remark}

\begin{figure}[t]
	\centering
	\resizebox{.8\textwidth}{!}{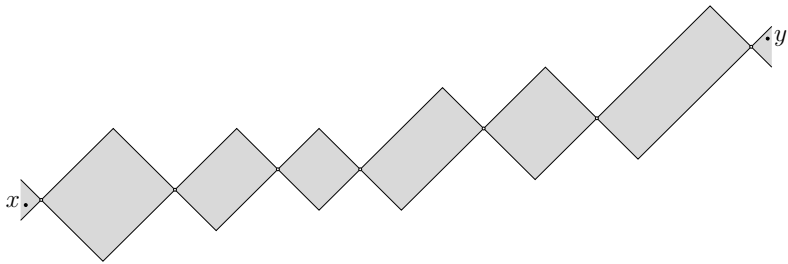}
	\caption{Typical realizations of the cluster \(C_{xy}\) of \(x\) and \(y\) under $\p_\beta(\cdot\given x \leftrightarrow y)$ (or of the HT random path connecting \(x\) and \(y\))
		can be decomposed into a concatenation of smaller pieces contained in ``diamonds'' (plus two boundary pieces contained, respectively, in \(\bcone_{x_0}\) and in \(\fcone_{x_M}\)).
	}
	\label{fig:EffRW}
\end{figure}

\smallskip
The same holds replacing the clusters by diamond-contained paths extracted from the HT representation, albeit with different measures \(\rhoL, \rhoR\) and \(p\). For simplicity, we will use the same notation in both cases, as the actual form of these measures plays no role in our analysis.

\smallskip
We will denote by \(\walkLaw{u}\) the distribution of the (directed) random walk \((\walk_k)_{k\geq 0}\) on \(\Zd\) with \(\walk_0=u\) and transition probabilities given by \(p(\cdot)\) (understood as a measure on the displacement). We will also write \(\walkLaw{u}^v\) the law of this walk conditioned on hitting \(v\).

As the walk \((\walk_k)_{k\geq 0}\) is directed, we will interpret the first coordinate axis as the ``time coordinate''. In this way, the walk becomes a space-time walk, and we will write $\walk_k=(\walk^{\parallel}_k,\walk^{\perp}_k)\in\bbZ\times\bbZ^{d-1}$.

The properties of the measure \(p\) guarantee that the increments of the random walk have exponential tails and that both the random walk \((\walk_k^\perp)_{k\geq 0}\) and the renewal process \((\walk_k^\parallel)_{k\geq 0}\) are aperiodic. Moreover, Property~\ref{it:finEnergEone} implies the irreducibility of \((\walk_k^\perp)_{k\geq 0}\) and the fact that \((\walk_k^\parallel)_{k\geq 0}\) can reach any time value with positive probability.

%
%

\subsubsection{Synchronized random walks}
For our purposes in this paper, we will need to consider two independent walks.
Let \(x,z\in\Zd\), with \(x\neq z\) and denote by \(\walkLaw{x\otimes z}\) the joint distribution of two independent random walks \(\walk\) and \(\walk'\) as above, starting respectively at \(x\) and \(z\). Denote also \(\walkLaw{x\otimes z}^{y\otimes w}\) the law of \(\walk,\walk'\) started at \(x\), resp.\ \(z\), and conditioned to hit \(y\), resp.\ \(w\).
It will be convenient to ``synchronize'' the two walks. Namely, let \(I=\setof{i\in\Z_{\geq 0}}{\exists j\in\Z_{\geq 0} \text{ s.t. } \walk_{i}^\parallel = (\walk'_{j})^\parallel}\) and
\(I'=\setof{j\in\Z_{\geq 0}}{\exists i\in\Z_{\geq 0} \text{ s.t. } \walk_{i}^\parallel = (\walk'_{j})^\parallel}\). We order the elements of \(I\) and \(I'\) into two increasing sequences \((i_k)_{k\geq 0}\) and \((j_k)_{k\geq 0}\). We can then define two new random walks by (see Fig.~\ref{fig:diffRW})
\[
\tilde{\walk}_k = \walk_{i_k},  \quad\text{ and }\quad   \tilde{\walk}'_k = \walk'_{j_k}.
\]

\medskip
Under \(\walkLaw{x\otimes z}^{y\otimes w}\), we will use the notation \(\#\) to denote the (random) number of steps of the finite trajectories of the synchronized walks. Notice that, by exponential decay of steps, there exists \(\delta>0\) such that this number is at least \(\delta\min(|z_1-x_1|,|y_1-w_1|)\) with probability at least \(1-e^{-c\min(|z_1-x_1|,|y_1-w_1|)}\) for some \(c>0\).

Moreover, we will write
\[
D(\tilde{\walk}) = \bigcup_{k=1}^{\#} \diam(\tilde{\walk}_{k-1},\tilde{\walk}_k)
\qquad\text{and}\qquad
D(\tilde{\walk}') = \bigcup_{k=1}^{\#} \diam(\tilde{\walk}'_{k-1},\tilde{\walk}'_k).
\]
Note that, by construction, the confinement property (Remark~\ref{rem:coneConfinement}) still holds for the synchronized random walks, that is, if \(C,C'\) denote the clusters marginal under \(\Xi_x^y\times\Xi_z^w \),
\begin{equation}\label{eq:DoubleConeConfinement}
C \subset \bcone_{\tilde{\walk}_0} \cup D(\tilde{\walk}) \cup \fcone_{\tilde{\walk}_{\#}}
\qquad\text{and}\qquad
C' \subset \bcone_{\tilde{\walk}'_0} \cup D(\tilde{\walk}') \cup \fcone_{\tilde{\walk}'_{\#}},
\end{equation} where \(\tilde{\walk},\tilde{\walk}'\) are the synchronization of the trajectories marginal of \(\Xi_x^y\times\Xi_z^w \).

\begin{figure}[t]
\resizebox{!}{5cm}{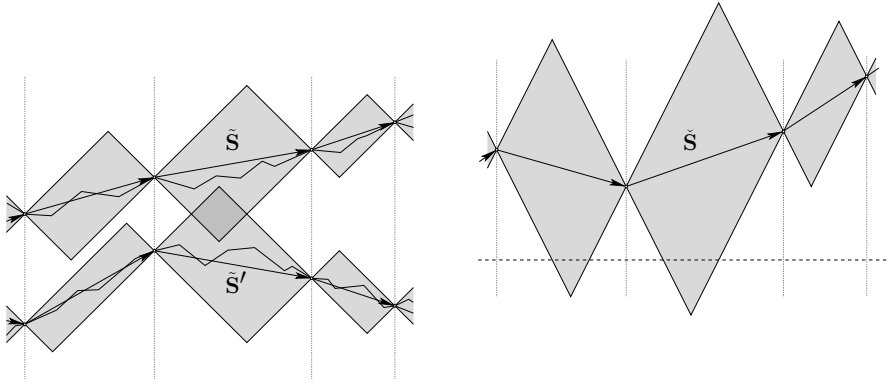}
\caption{Left: The synchronized random walks \(\tilde{\walk}\) and \(\tilde{\walk}'\), drawn with the associated diamonds containing the relevant microscopic object. Right: the corresponding difference walk \(\check{\walk}\). Observe that a necessary, but not sufficient, condition for two of the original diamonds to intersect is that the corresponding enlarged diamond associated to \(\check{\walk}\) (corresponding to the value \(2\delta\)) intersects the first coordinate axis (represented by the dashed line). (The picture is two-dimensional, but this observation is true in general.)}
\label{fig:diffRW}
\end{figure}

\subsubsection{Difference random walk}\label{sec:DiffRW}
Let us denote by \(\diffSyncWalkLaw_u\) the distribution of the random walk \((\check{\walk}_k)_{k\geq 0}\) defined by \(\check{\walk}_k^\perp = \tilde{\walk}_k^\perp-(\tilde\walk'_k)^\perp\) and \(\check{\walk}_k^\parallel = \tilde{\walk}_k^\parallel = (\tilde\walk'_k)^\parallel\), starting at \(\tilde{\walk}_0-\tilde\walk'_0=u\in\Zd\), and by \(\check{X}_k\) its increments (see Fig.~\ref{fig:diffRW}). It then follows from the above properties that (see~\cite{IV2012} for a similar construction)
\begin{enumerate}
	\item\label{DRW-DecayProp} \(\exists c>0\) such that, \(\forall x,z\in\Zd\), \(\walkLaw{x\otimes z}(\norm{\check\walk_0-(x^\parallel\vee z^\parallel, x^\perp-z^\perp)} \geq t) \leq e^{-c t}\);
	\item \((\check X_k)_{k\geq 1}\) are i.i.d.\ random variables with exponential tail;
	\item \((\check{X}_k^\parallel,\check{X}_k^\perp) \stackrel{\text{law}}{=} (\check{X}_k^\parallel,-\check{X}_k^\perp)\);
	\item\label{DRW-ConeProp} \(\check{X}_k^\parallel\geq\tfrac\delta2\norm{\check{X}_k^\perp}\) almost surely;
	\item \((\check{\walk}_k^{\perp})_{k\geq 0}\) is irreducible and aperiodic and \((\check{\walk}_k^{\parallel})_{k\geq 0}\) is aperiodic and can reach any times larger that its starting time with positive probability.
\end{enumerate}
To shorten notation, we will write \(\diffSyncWalkLaw_{x}\equiv\diffSyncWalkLaw_{(0,x)}\) with \(x\in\Z^{d-1}\).

\subsubsection{Some notations}\label{sec:NotationsRW}
Let
\[
\tau_0^\perp = \inf\setof{k\geq 1}{\diffSyncWalk_k^\perp=0}
\qquad
\text{ and }
\qquad
\tau_n^\parallel = \inf\setof{k\geq 1}{\diffSyncWalk_k^\parallel\geq n} .
\]
We first introduce a few events that will be important in our analysis. Given \(n\in\bbZ_{\geq 0}\) and \(y\in\bbZ^{d-1}\), we set
\begin{gather*}
\calP_n(y) = \{\exists k\geq 0:\; \diffSyncWalk_k=(n,y)\},\\
\calQ_n(y) = \{\exists k\geq 0:\; \diffSyncWalk_k=(n,y) \text{ and } \tau_0^\perp \geq k \},\\
\calR_n = \{\diffSyncWalk_{\tau_0^\perp}^{\parallel} \geq n \}.
\end{gather*}
The corresponding probabilities will be denoted
\begin{gather*}
p_n(x,y) = \diffSyncWalkLaw_{x}\bigl( \calP_n(y) \bigr),\qquad
q_n(x,y) = \diffSyncWalkLaw_{x}\bigl( \calQ_n(y) \bigr),\qquad
r_n(x) = \diffSyncWalkLaw_{x}\bigl( \calR_n \bigr).
\end{gather*}
The asymptotic behavior of these quantities (as \(n\to\infty\), with \(x=y=0\)) is discussed in Appendix~\ref{app:RW}.

\smallskip
Let us stress that, in the above definitions, the index \(n\) denotes a distance along the first coordinate axis and not a number of steps. This creates some (minor) complications, and we will need to pass from one description to the other. For this reason, given $\ell\in \Hrange = \setof{\diffSyncWalk_k^{\parallel}}{k\geq 0}$, define $t_\ell$ to be such that $\diffSyncWalk_{t_\ell}^\parallel=\ell$.

\smallskip
Finally, let
\[
\psi_d(n) =
\begin{cases}
	n^{-1}				& \text{when } d=2,\\
	\log(n+1)^{-2}		& \text{when } d=3,\\
	1					& \text{when } d\geq 4,
\end{cases}
\]
and set \(\phi_d(n) = n^{-(d-1)/2} \psi_d(n)\).

%
%

\subsection{Proof of the upper bound}

In this section, we prove the upper bound in Theorem~\ref{thm:OZ-even-even}.
To this end, we will use the RC and the HT representations together with the associated coupling to a directed random walk.

We start by deriving an upper bound in terms of an event involving two independent realizations of the model.
\begin{lemma}
	\label{lem:UBdecoupledRW}
	\begin{equation}
		\label{eq:UBstart}
		\isingLaw{\sigma_\nA;\sigma_\nB}{}
		\leq
		\sum_{\substack{A_1\subset \nA,B_1\subset \nB\\|A_1|,|B_1| \text{\rm\ odd}}}
		\isingLaw{\sigma_{A_1}\sigma_{B_1}}{} \isingLaw{\sigma_{A_1^\comp}\sigma_{B_1^\comp}}{} \,
		\currentLaw{A_1\cup B_1\otimes A_1^\comp\cup B_1^\comp}{G}(A_1\nleftrightarrow A_1^\comp,B_1\nleftrightarrow B_1^\comp),
	\end{equation}
	where \(A_1^\comp = \nA\setminus A_1\), \(B_1^\comp = \nB\setminus B_1\) and the notation \(C \nleftrightarrow D\) means that the two sets \(C,D\subset V_G\) are not connected  in \(\widehat{\current_1+\current_2}\).
\end{lemma}
\begin{proof}
	Let us write \(\frD = \{A_1\nleftrightarrow A_1^\comp, B_1\nleftrightarrow  B_1^\comp\}\).
	First recall that, since \(A\cap\nB=\emptyset\) once \(n\) is large enough, it follows from~\eqref{eq:CovABRC} that
	\[
	\isingLaw{\sigma_\nA;\sigma_\nB}{}
	=
	\frac{\currentPF{\nA\cup \nB}{} \currentPF{\emptyset}{} \{\EvPart{A}^{\,\comp}\}}{\currentPF{\emptyset}{}^2}.
	\]
	To get an upper bound, simply notice that the event $\EvPart{A}^{\,\comp}$ and the constraint
	\(\partial\current_1 = \nA\cup\nB\) imply that one can find two sets $A_1\subset \nA, B_1\subset \nB$ of odd cardinality such that \(\EvPart{A_1\cup B_1}\), \(\EvPart{A_1^\comp\cup B_1^\comp}\), \(A_1\nleftrightarrow A_1^\comp\) and \(B_1\nleftrightarrow B_1^\comp\) are all realized in \(\widehat{\current_1+\current_2}\). Thus,
	\begin{align*}
	\frac{1}{\currentPF{\emptyset}{}^2}& \currentPF{\nA\cup \nB}{} \currentPF{\emptyset}{} \{\EvPart{A}^{\,\comp}\}\notag\\
	&\leq
	\frac{1}{\currentPF{\emptyset}{}^2} \sum_{\substack{A_1\subset \nA\\|A_1|\text{\rm\ odd}}} \sum_{\substack{B_1\subset \nB\\|B_1|\text{\rm\ odd}}}
	\currentPF{\nA\cup \nB}{} \currentPF{\emptyset}{} \{\EvPart{A_1\cup B_1} \cap \EvPart{A_1^\comp\cup B_1^\comp} \cap \frD\}\notag\\
	&=
	\frac{1}{\currentPF{\emptyset}{}^2} \sum_{\substack{A_1\subset \nA\\|A_1|\text{\rm\ odd}}} \sum_{\substack{B_1\subset \nB\\|B_1|\text{\rm\ odd}}}
	\currentPF{A_1\cup B_1}{} \currentPF{A_1^\comp\cup B_1^\comp}{} \{\frD\}\notag\\
	&=
	\sum_{\substack{A_1\subset \nA\\|A_1|\text{\rm\ odd}}} \sum_{\substack{B_1\subset \nB\\|B_1|\text{\rm\ odd}}}
	\isingLaw{\sigma_{A_1}\sigma_{B_1}}{G} \, \isingLaw{\sigma_{A_1^\comp}\sigma_{B_1^\comp}}{G} \,
	\currentLaw{A_1\cup B_1\otimes A_1^\comp\cup B_1^\comp}{G}(\frD),
	\end{align*}
	where the first equality is again obtained via the Switching Lemma.
	\end{proof}

In view of Remark~\ref{rem:couplingRCHT}, we can couple the double random-current measure with a double HT measure to obtain
	\begin{equation}\label{ineq:crucial1}
		\currentLaw{A_1\cup B_1\otimes A_1^\comp\cup B_1^\comp}{G}(A_1\nleftrightarrow A_1^\comp, B_1\nleftrightarrow B_1^\comp)
		\leq
		\evGraphLaw{A_1\cup B_1\otimes A_1^\comp\cup B_1^\comp}{G}(A_1\nleftrightarrow A_1^\comp, B_1\nleftrightarrow B_1^\comp),
	\end{equation}
where the absence of connexions in the last expression is with respect to the union of the two HT configurations.

Now, observe that the realization of \(\{A_1\nleftrightarrow A_1^\comp, B_1\nleftrightarrow B_1^\comp\}\) under \(\evGraphLaw{A_1\cup B_1\otimes A_1^\comp\cup B_1^\comp}{G}\) entails the existence of at least one path from \(A_1\) to \(B_1\) in the first copy of the process and at least one path from \(A_1^\comp\) to \(B_1^\comp\) in the second copy. We thus obtain, using Lemma~\ref{lem:HTPathMonotonicity},
\begin{align*}
	\evGraphLaw{A_1\cup B_1\otimes A_1^\comp\cup B_1^\comp}{G}(A_1\nleftrightarrow A_1^\comp, B_1\nleftrightarrow B_1^\comp)
	\leq
	\sum_{\substack{x\in A_1\\y\in A_1^\comp}}
	\sum_{\substack{u\in B_1\\v\in B_1^\comp}}
	\randPathLaw^{\{x,u\}\otimes \{y,v\}}_G (\gamma \cap \gamma' = \emptyset),
\end{align*}
where \(\gamma\) denotes the path connecting \(x\) to \(u\) in the first copy, \(\gamma'\) the path connecting \(y\) to \(v\) in the second copy, and the event \(\{\gamma\cap\gamma'=\emptyset\}\) means that these two paths have no vertex in common.

The expectations \(\isingLaw{\sigma_{A_1}\sigma_{B_1}}{G}\) and \(\isingLaw{\sigma_{A_1^\comp}\sigma_{B_1^\comp}}{G}\) can be estimated using Theorem~\ref{thm:OZ-odd-odd}.
The proof will therefore be complete once we show that there exists $\Cl{UBcst}$, depending on \(A,B,\beta,d\) but not on \(n\), such that
\begin{equation}\label{eq:ProbDisjointPaths}
	\randPathLaw^{\{x,u\}\otimes \{y,v\}}_G (\gamma \cap \gamma' = \emptyset)
	\leq
	\Cr{UBcst}\psi_d (n),
\end{equation}
for all \(n\) large enough, uniformly in $x\in A_1, y\in A_1^\comp, u\in B_1, v\in B_1^\comp$.
Note that \(\psi_d (n)\) corresponds to the behavior of the non-intersection probability for two independent directed random walks on \(\Zd\) conditioned to start at $(0,0)$ and end at $(n,0)$ (see Appendix~\ref{app:RW}).

The bound~\eqref{eq:ProbDisjointPaths} is clearly trivial when $d\geq 4$. For $d=2,3$, consider the directed $d$-dimensional walk $\diffSyncWalk$ introduced in Section~\ref{sec:RWrep}. We start with a lemma which is a straightforward consequence of the discussion in Appendix~\ref{app:RW}.
\begin{lemma}
	\label{lem:UBnoReturn}
	There exists $\Cl{UBnoRet}$ such that, for any \(z,w\in\Z^{d-1}\) and any \(m\geq 1\),
	\begin{align}
	\label{eq:UBqnxy}
		q_m(z,w)\leq \Cr{UBnoRet} (1+\normII{z})^{d+1} (1+\normII{w})^{d+1} \phi_d(m).
	\end{align}
	In particular, there exists $\Cl{UBcondnoRet}$ such that, whenever $\normII{z},\normII{w}\leq \sqrt{m}$,
	\begin{align}
		\diffSyncWalkLaw_{z} \bigl( \calR_m \given \calP_m(w) \bigr) \leq \Cr{UBcondnoRet} (1+\normII{z})^{d+1} (1+\normII{w})^{d+1} \psi_d(m).
	\end{align}
\end{lemma}
\begin{proof}
	First, notice that~\eqref{eq:LBp} entails the existence of $c>0$ such that, for any $r\in\Z^{d-1}$,
	\[
		p_{\normII{r}^2}(0,r) \geq c (1+\normII{r})^{-(d-1)}.
	\]
	Therefore, uniformly in \(w,z\in\Z^{d-1}\),
	\begin{align*}
		q_m(z,w) \leq  c\, (1+\normII{z})^{d-1} (1+\normII{w})^{d-1} p_{\normII{z}^2}(0,z)\, p_{\normII{w}^2}(0,w)\, q_m(z,w).
	\end{align*}
	Now, using the decomposition
	\[
		p_{m}(0,z) = \sum_{t=0}^{m} p_{m-t}(0,0) q_t(0,z),
	\]
	one obtains from Proposition~\ref{pro:AppRW} and~\eqref{eq:UBp} that
	\begin{align*}
		p_{\normII{z}^2}(0,z)\, p_{\normII{w}^2}(0,w)\,& q_m(z,w)\\
		&\leq
		\sum_{l=0}^{\normII{z}^2} \sum_{r=0}^{\normII{w}^2} p_{\normII{z}^2-l}(0,0)\, p_{\normII{w}^2-r}(0,0)\, \diffSyncWalkLaw_{0}(\tau_{0}^{\perp}= m+r+l)\\
		&\leq
		c \sum_{l=0}^{\normII{z}^2} \sum_{r=0}^{\normII{w}^2}
		\begin{cases}
			(m+r+l)^{-3/2}					& \text{when } d=2\\
			(m+r+l)^{-1}\log(m+r+l)^{-2}	& \text{when } d=3
		\end{cases}\\
		&\leq
		c\, (1+\normII{z})^2 (1+\normII{w})^2
		\begin{cases}
			m^{-3/2}				& \text{when } d=2\\
			m^{-1}\log(m)^{-2}		& \text{when } d=3.
		\end{cases}
	\end{align*}
	This gives equation~\eqref{eq:UBqnxy}. Since, by~\eqref{eq:LBp},
	\(
		p_m(z,w) \geq \Cl{CLTcstLB} m^{-(d-1)/2}
	\)
	when \(\normII{z},\normII{w}\leq \sqrt{m}\), the second claim also immediately follows:
	\begin{align*}
		\diffSyncWalkLaw_{z} (\calR_m \given \calP_m(w))
		&=
		\frac{q_m(z,w)}{p_m(z,w)} \\
		&\leq
		\Cr{UBcondnoRet} (1+\normII{z})^{d+1} (1+\normII{w})^{d+1}
		\begin{cases}
			m^{-1} 			& \text{when } d=2,\\
			\log(m)^{-2} 	& \text{when } d=3.
		\end{cases}
	\end{align*}
\end{proof}

We will now use the coupling of HT-paths with directed random walks. Let us write \(\zeta(\ell,L,z,z',w,w') = \Xi_{x}^{u}\times\Xi_{y}^{v}\bigl( \tilde{\walk}_0=(\ell,z+z'), \tilde{\walk}'_0=(\ell,z'), \tilde{\walk}_\#=(\ell+L,w+w'), \tilde{\walk}'_\#=(\ell+L,w') \bigr)\). Then,
\[
	\randPathLaw^{\{x,u\}\otimes \{y,v\}}_G (\gamma \cap \gamma' = \emptyset)
	\leq
	\sum_{\substack{\ell,L\in\Z_{\geq 0}\\z,z',w,w'\in\Z^{d-1}}} \!\!\!\!
	\zeta(\ell,L,z,z',w,w') \,
	\syncWalkLaw{z} \bigl( \calR_L \bgiven \calP_L(w) \bigr) + e^{-\Cr{expDecTVdist}n},
\]
since \(\gamma\cap\gamma'=\emptyset\) implies that \(\tilde{\walk}_k\neq\tilde{\walk}'_k\) for all \(0\leq k\leq\#\), whenever (the HT path version of) \eqref{eq:DoubleConeConfinement} applies. The properties of the measure \(\Xi\) guarantee that \(\Xi(\tilde{\walk}_\#^\parallel - \tilde{\walk}_0^\parallel < n/2) \leq e^{-cn}\) for some \(c>0\), once \(n\) is large enough. We can therefore assume that \(L\geq n/2\). In that case, Lemma~\ref{lem:UBnoReturn} implies that
\[
	\diffSyncWalkLaw_{z} \bigl( \calR_L \bgiven \calP_L(w) \bigr)
	\leq
	c\, (1+\normII{z})^{d+1} (1+\normII{w})^{d+1} \psi_d(n).
\]
The conclusion now follows from Property~\eqref{DRW-DecayProp} in Section~\ref{sec:DiffRW} and the exponential tails of \(\rho_{\scriptscriptstyle\rm L}\) and \(\rho_{\scriptscriptstyle\rm R}\).

%
%

\subsection{Proof of the lower bound}

In this section, we prove the lower bound of Theorem~\ref{thm:OZ-even-even}.
As for the upper bound, the first step will be to reduce the analysis to two independent walk-like objects conditioned on not intersecting.
This time, the random-cluster representation provides the adequate setting.
\begin{lemma}
	\label{lem:LBdecoupledRW}
	Assume that \(n\) is sufficiently large to ensure that \(A\cap\nB=\emptyset\).
	For any $x,y\in \nA$ and $u,v\in \nB$ with $x\neq y$ and $u\neq v$, the following bound holds:
	\begin{multline}
		\frac{\isingLaw{\sigma_\nA;\sigma_\nB}{G}}{\isingLaw{\sigma_x\sigma_u}{G}\isingLaw{\sigma_y\sigma_v}{G}} \geq
		\sum_{\substack{C_1\ni x,u\\C_2\ni y,v\\C_1\cap C_2=\emptyset}} \RCMLaw{\EvPart{\nA\cup \nB} \cap \EvPart{A}^{\,\comp} \given C_{x,u}=C_1, C_{y,v}=C_2}{G} \times\\
		\times \RCMLaw{C_{x,u}=C_1 \given x\leftrightarrow u}{G} \RCMLaw{C_{y,v}=C_2 \given y\leftrightarrow v}{G}.
		\label{eq:MainEstimateLB}
	\end{multline}
\end{lemma}
\begin{proof}
	Given a set $E\subset E_G$ of edges, we will consider below the two FK events
	\(
		\openSet{E}=\{E \open\} = \{\omega_e=1\ \forall e\in E \}
	\)
	and
	\(
		\closeSet{E}=\{E \close\}= \{\omega_e=0\ \forall e\in E \}.
	\)

	We start with the FK representation of $\isingLaw{\sigma_\nA;\sigma_\nB}{G}$.
	Let $x,y\in\nA$, $x\neq y$, and $u,v\in \nB$, $u\neq v$.
	Using~\eqref{eq:FKsigmaA}, we can write
	\begin{align}
		\isingLaw{\sigma_\nA;\sigma_\nB}{G}
		&=
		\RCMLaw{\EvPart{\nA\cup \nB}}{G} - \RCMLaw{\EvPart{\nA}}{G} \RCMLaw{\EvPart{\nB}}{G}\notag\\
		&=
		\RCMLaw{\EvPart{\nA\cup \nB} \cap \EvPart{A}^{\,\comp}}{G} + \RCMLaw{\EvPart{\nA\cup \nB} \cap \EvPart{\nA}}{G} - \RCMLaw{\EvPart{\nA}}{G} \RCMLaw{\EvPart{\nB}}{G}\notag\\
		&=
		\RCMLaw{\EvPart{\nA\cup \nB} \cap \EvPart{A}^{\,\comp}}{G} + \RCMLaw{\EvPart{\nA} \cap \EvPart{\nB}}{G} - \RCMLaw{\EvPart{\nA}}{G} \RCMLaw{\EvPart{\nB}}{G}\notag\\
		&
		\geq \RCMLaw{\EvPart{\nA\cup \nB} \cap \EvPart{A}^{\,\comp}}{G}
		\label{ineq:crucial2}
		\\
		&\geq \RCMLaw{\EvPart{\nA\cup \nB} \cap \EvPart{A}^{\,\comp}, x\leftrightarrow u, y\leftrightarrow v, x\nleftrightarrow y}{G},\notag
	\end{align}
	where the first inequality follows from the FKG inequality, observing that $\EvPart{A}$ and $\EvPart{\nB}$ are increasing events.

	Next, we partition the event in the last expression according to the realizations of the two induced clusters, that is, we sum over all pairs of clusters \(C_1,C_2\) such that \(x,u\in C_1\), \(y,v\in C_2\) and \(C_1\cap C_2=\emptyset\).
	\begin{multline*}
	\RCMLaw{\EvPart{\nA\cup \nB} \cap \EvPart{A}^{\,\comp}, x\leftrightarrow u, y\leftrightarrow v, x\nleftrightarrow y}{G}\\
	=
	\sum_{\substack{C_1,C_2\\C_1\cap C_2=\emptyset}}  \RCMLaw{\EvPart{\nA\cup \nB} \cap \EvPart{A}^{\,\comp} \given C_{x,u}=C_1, C_{y,v}=C_2}{G} \RCMLaw{C_{x,u}=C_1, C_{y,v}=C_2}{G}.
	\end{multline*}
	Writing, for \(E\subset E_G\), \(\partial E = \setof{e=\{i,j\}\in E_G}{i\in E, j\notin E}\), we obtain, using again the FKG inequality,
	\begin{align*}
	\RCMLaw{C_{x,u}=C_1, C_{y,v}=C_2}{G}\hspace{-2.5cm}&\\
	&=
	\RCMLaw{\openSet{C_1}, \closeSet{\partial C_1}, \openSet{C_2},\closeSet{\partial C_2}}{G}\\
	&=
	\RCMLaw{\openSet{C_1}, \closeSet{\partial C_1} \given  \closeSet{\partial C_2}}{G}\RCMLaw{C_{y,v}=C_2}{G}\\
	&=
	\RCMLaw{\openSet{C_1} \given  \closeSet{\partial C_1}}{G} \RCMLaw{ \closeSet{\partial C_1} \given \closeSet{\partial C_2}}{G} \RCMLaw{C_{y,v}=C_2}{G}\\
	&\geq
	\RCMLaw{C_{x,u}=C_1}{G} \RCMLaw{C_{y,v}=C_2}{G}\\
	&=
	\RCMLaw{x\leftrightarrow u}{G} \RCMLaw{y\leftrightarrow v}{G} \RCMLaw{C_{x,u}=C_1 \given x\leftrightarrow u}{G}\\
	&\hspace{5cm}\times
	\RCMLaw{C_{y,v}=C_2 \given y\leftrightarrow v}{G}.\qedhere
	\end{align*}
\end{proof}
Again, the expectations \(\isingLaw{\sigma_x\sigma_u}{G}\) and \(\isingLaw{\sigma_y\sigma_v}{G}\) can be estimated by Theorem~\ref{thm:OZ-odd-odd}, so that we only have to control the right-hand side of~\eqref{eq:MainEstimateLB}.

In what follows, we suppose \(n\) large enough for the definitions to make sense. Define \(K=\max_{x\in A\cup B}\normI{x}\). One has \(A\subset\Z_{\leq K}\times \Z^{d-1}\) and \(\nB\subset\Z_{\geq (n\uvec)_1-K}\times \Z^{d-1} \). Let \(a_1=(10K,2K,0,\dots,0)\), \(a_2=(10K,-2K,0,\dots,0)\) and \(b_1=n\uvec-(10K,-2K,0,\dots,0)\), \(b_2=n\uvec-(10K,2K,0,\dots,0)\).

We then use finite energy to prove the following:
\begin{lemma}
	\label{lem:LB_finiteEnergy}
	One can construct \(H\subset\{C\ni x,u\}\times\{C\ni y,v\}\) satisfying:
	\begin{itemize}
		\item there exists \(c>0\) not depending on \(n\) such that
		\[
		\min_{(C_1,C_2)\in H} \RCMLaw{\EvPart{\nA\cup \nB} \cap \EvPart{A}^{\,\comp} \given C_{x,u}=C_1, C_{y,v}=C_2}{G}\geq c.
		\]
		\item There exist \(c,C>0\) not depending on \(n\), such that
		\begin{multline*}
		\sum_{\substack{(C_1,C_2)\in H\\ C_1\cap C_2=\varnothing}}\RCMLaw{C_{x,u}=C_1 \given x\leftrightarrow u}{G} \RCMLaw{C_{y,v}=C_2 \given y\leftrightarrow v}{G}\geq\\
		\geq -e^{-cn} + C\diffSyncWalkLaw_{a_1-a_2} \bigl(\norm{\diffSyncWalk^{\perp}_k} > 2\delta\check{X}_{k+1}^{\parallel}, \; \forall k<\tau^{\parallel}_{L} \bgiven \calP_L(b_1-b_2) \bigr),
		\end{multline*}
		where \(L=(n\uvec)_1-20K\).
	\end{itemize}
\end{lemma}
\begin{proof}
	We start by constructing \(H\).
	\begin{figure}[h]
		\centering
		\includegraphics[scale=0.8]{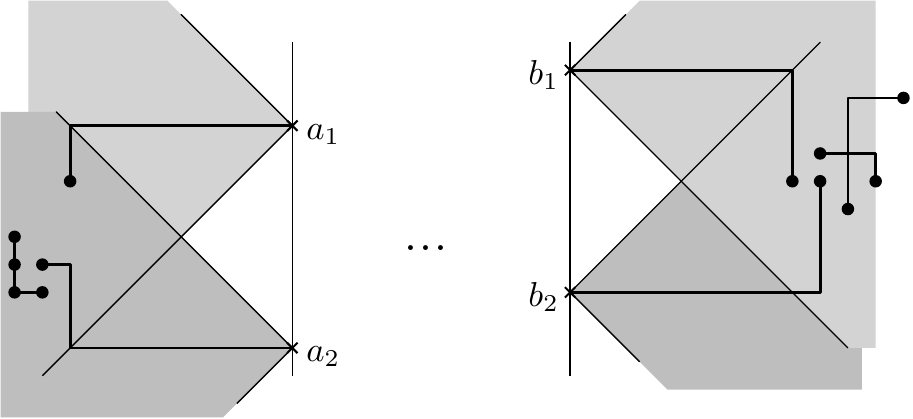}
		\caption{An example of admissible boundary conditions.
		}
		\label{fig:FinEnergAB}
	\end{figure}

	There exist \(x,y\in A\) (see Figure~\ref{fig:FinEnergAB}) such that one can find
	\begin{itemize}
		\item a self-avoiding path \(\gamma_{xa_1}\) connecting \(x\) to \(a_1\) included in \(\bcone_{a_1}\),
		\item a self-avoiding path \(\gamma_{ya_2}\) connecting \(y\) to \(a_2\) included in \(\bcone_{a_2}\),
		\item a cluster \(C_{A\setminus\{x,y\}}\) connecting all sites of \(A\setminus\{x,y\}\) together included in \([-K,K]^d\),
	\end{itemize}
	with \(\gamma_{xa_1},\gamma_{ya_2},C_{A\setminus\{x,y\}}\) all disjoints, in particular, the edge boundary of any of those three items does not intersect the other items.
	Fix a compatible 5-tuple \(x,y,\gamma_{xa_1},\gamma_{ya_2},C_{A\setminus\{x,y\}}\). In the same fashion, there exist \(u,v\in \nB\) (see Figure~\ref{fig:FinEnergAB}) such that one can find
	\begin{itemize}
		\item a self-avoiding path \(\gamma_{ub_1}\) connecting \(u\) to \(b_1\) included in \(\fcone_{b_1}\),
		\item a self-avoiding path \(\gamma_{vb_2}\) connecting \(v\) to \(b_2\) included in \(\fcone_{b_2}\),
		\item a cluster \(C_{\nB\setminus\{u,v\}}\) connecting all sites of \(\nB\setminus\{u,v\}\) together, included in \(n\uvec+[-K,K]^d\),
	\end{itemize} with \(\gamma_{ub_1},\gamma_{vb_2},C_{\nB\setminus\{u,v\}}\) all disjoints. Fix a compatible 5-tuple \(u,v,\gamma_{ub_1},\gamma_{vb_2},C_{\nB\setminus\{u,v\}}\).

	We then define \( H\subset\{C\ni x,u\}\times\{C\ni y,v\} \) the set of pairs of cluster \((C_1,C_2)\) such that
	\begin{itemize}
		\item \(C_1\cap \bcone_{a_1}=\gamma_{xa_1},\ C_1\cap \fcone_{b_1}=\gamma_{ub_1},\ C_1\setminus(\gamma_{xa_1}\cup \gamma_{ub_1})\subset\diam(a_1,b_1) \),
		\item \(C_2\cap \bcone_{a_2}=\gamma_{ya_2},\ C_2\cap \fcone_{b_2}=\gamma_{vb_2},\ C_2\setminus(\gamma_{ya_2}\cup \gamma_{vb_2})\subset\diam(a_2,b_2) \).
	\end{itemize}

	Notice that, for any \((C_1,C_2)\in H\), opening all edges of \(C_{\nB\setminus\{u,v\}}\cup C_{A\setminus\{x,y\}}\) and closing the edge-boundary of this set has probability bounded from below under \(\RCMLaw{\cdot \given C_{x,u}=C_1, C_{y,v}=C_2}{G}\) uniformly in \(n\) as \(C_{\nB\setminus\{u,v\}}\cup C_{A\setminus\{x,y\}}\) does not intersect \(C_1\cup\partial C_1\cup C_2\cup\partial C_2\) and the size of the support of this event is bounded uniformly in \(n\). The first part of the lemma thus follows from the fact that this event is a sub-event of \(\EvPart{\nA\cup \nB} \cap \EvPart{A}^{\,\comp}\) under \(C_{x,u}=C_1, C_{y,v}=C_2\).

	The second part uses the replacement of \(\RCMLaw{\cdot \given x\leftrightarrow u}{G}\) by \(\Xi_x^u\) explained in Section~\ref{sec:RWrep} (the same replacement is done for \(\RCMLaw{\cdot \given y\leftrightarrow v}{G}\) by \(\Xi_y^v\)). By~\ref{it:finEnergBC} and finite energy (opening the edges of \(\gamma_{xa_1}\) and closing the edges of \(\partial \gamma_{xa_1}\) has a probability bounded from below under \(\RCMLaw{\cdot \given x\leftrightarrow u}{G}\)), one has: there exists \(c>0\) such that
	\[
		\Xi_x^u(B_L=\gamma_{xa_1},B_R=\gamma_{ub_1}) \geq c.
	\]
	So,
	\begin{align*}
		\sum_{\substack{(C_1,C_2)\in H\\ C_1\cap C_2=\varnothing}}& \RCMLaw{C_{x,u}=C_1 \given x\leftrightarrow u}{G} \RCMLaw{C_{y,v}=C_2 \given y\leftrightarrow v}{G} \geq\\
		&\geq
		-e^{-cn} + \Xi_x^u(B_L=\gamma_{xa_1},B_R=\gamma_{ub_1})\,\Xi_y^v(B_L=\gamma_{ya_2},B_R=\gamma_{vb_2})\\
		&\hspace{5.5cm}\times\ \walkLaw{a_1\otimes a_2}^{b_1\otimes b_2}\bigl(\diam(\walk)\cap\diam(\walk')=\varnothing\bigr)\\
		&\geq-e^{-cn} +C\diffSyncWalkLaw_{a_1-a_2} \bigl(\norm{\diffSyncWalk^{\perp}_k} > 2\delta\check{X}_{k+1}^{\parallel}, \; \forall k<\tau^{\parallel}_{L} \bgiven \calP_L(b_1-b_2) \bigr) ,
	\end{align*}
	where \(L= (n\uvec)_1-20K\). The first inequality is the total variation estimate~\ref{it:expDecTVdist}, inclusion of events and Remark~\ref{rem:coneConfinement}, the second is inclusion of events (recall that \(\diffSyncWalkLaw_{a_1-a_2}\) is the law of the difference walk induced by the two synchronized walks obtained from \(\Xi_x^u\) and \(\Xi_y^v\); by construction, the two walks have synchronized starting time).
\end{proof}

The lower bound in Theorem~\ref{thm:OZ-even-even} will follow by plugging the results of Lemma~\ref{lem:LB_finiteEnergy} in \eqref{eq:MainEstimateLB} and direct use of the following lemma.
\begin{lemma}
	\label{lem:noDiamIntersecLB}
	For any $C>0$, there exists $\Cl{LBcst}>0$ such that, for any $z,w\in\Z^{d-1}$ with $\normII{z},\normII{w}\leq C$ and all \(m\) large enough, one has
	\begin{equation}
		\diffSyncWalkLaw_{z}\bigl(\norm{\diffSyncWalk_i^{\perp}} \geq 2\delta\check{X}_{i+1}^{\parallel}\; \forall i<\tau^{\parallel}_{m} \bgiven \calP_m(w) \bigr)
		\geq
		\Cr{LBcst} \psi_d(m).
	\end{equation}
\end{lemma}
\begin{proof}
	The proof is done in two steps. In the first one, we prove that
	\begin{claim}
		\label{cl:LB1}
		For any \(z,w\in\Z^{d-1}\), there exists $c>0$ such that, for all \(m\) large enough,
		\begin{equation}
			\label{eq:noReturnLB}
			\diffSyncWalkLaw_{z}(\calQ_m(w) \given \calP_m(w))
			\geq
			c \,\psi_d(m).
		\end{equation}
	\end{claim}
	This is slight extension of the estimates of Proposition~\ref{pro:AppRW} to different starting and ending points than $0$.

	The second step will consist in showing that the walk conditioned not to hit zero will typically stay away from zero, so that Lemma~\ref{lem:noDiamIntersecLB} actually follows from Claim~\ref{cl:LB1}.

	\medskip
	We first show Claim~\ref{cl:LB1}.
	Let us fix, \(\ell_z,\ell_w\) such that	\(q_{\ell_z}(z,0) > 0\) and \(q_{\ell_w}(0,w) > 0\). Now, observe that
	\[
		q_m(z,w) \geq q_{\ell_z}(z,0) q_{m-\ell_z-\ell_w}(0,0) q_{\ell_w}(0,w).
	\]
	Indeed, under \(\diffSyncWalkLaw_{z}\), any trajectory contributing to \(\calP_m(w)\), visiting \((\ell_z,0)\) and \((m-\ell_w,0)\), but otherwise avoiding \(0\), can be transformed into a trajectory contributing to \(\calQ_m(w)\) by interchanging the two increments incident to \((\ell_z,0)\) and doing the same to those incident to \((m-\ell_w,0)\), without changing the probability of the trajectory.
	It thus follows from the lower bounds on \(f_m\) in Appendix~\ref{app:RW} (see~\eqref{eq:LBf3d}, \eqref{eq:LBf4d} and~\eqref{eq:LBf2d}) that there exists \(c(z,w)>0\) such that
	\begin{equation}\label{eq:LBqxy}
	q_{m}(z,w) \geq c(z,w)\, \phi(m),\qquad\text{for all \(m\) large enough.}
	\end{equation}
	\eqref{eq:noReturnLB} then follows from~\eqref{eq:UBp}.

	\medskip
	Let us now turn to the second step and show how Claim~\ref{cl:LB1} implies Lemma~\ref{lem:noDiamIntersecLB}.
	We introduce
	\[
		\calD_i
		=
		\begin{cases}
			\{\norm{\diffSyncWalk_{t_i}^{\perp}} < 2\delta\check{X}_{t_i+1}^{\parallel}\} 	& \text{if } i\in\Hrange,\\
			\emptyset 																		& \text{otherwise.}
		\end{cases}
	\]
	Note that, for \(\diam(\syncWalk_{t_i},\syncWalk_{t_{i+1}})\cap\diam(\syncWalk'_{t_i},\syncWalk'_{t_{i+1}})\neq\emptyset\) to occur, it is necessary that \(\calD_i\) occurs.
	To conclude, one thus has to show that there exists $c>0$, depending on \(z\) and \(w\), such that,  for all \(m\) large enough,
	\begin{align}
		\label{eq:not0toAway0}
		\diffSyncWalkLaw_{z} \bigl( \calD_{i}^\comp\; \forall i\in\Hrange\cap[0,m] \bgiven \calQ_{m}(w) \bigr) \geq c.
	\end{align}
	To this end, we separate the treatment of the trajectory close to the starting and ending points from the treatment away from them. Let $T>0$ be a large integer (which will be chosen later as a function of $z,w$) and write $I_T=[T,m-T]$. Then,
	\begin{align*}
		\diffSyncWalkLaw_{z} \bigl( \calD_{i}^\comp\; \forall i\in&\Hrange\cap[0,m] \bgiven \calQ_{m}(w) \bigr)\\
		&=
		\diffSyncWalkLaw_{z}\Bigl(\bigcap_{i\in\Hrange\cap I_T} \calD_{i}^\comp \Bgiven \calQ_{m}(w) \Bigr)
		\diffSyncWalkLaw_{z}\Bigl(\bigcap_{i\in\Hrange\cap[0,m]} \calD_{i}^\comp \Bgiven \calQ_{m}(w), \bigcap_{i\in\Hrange\cap I_T} \calD_{i}^\comp \Bigr)\\
		&\geq
		e^{-\Cl{step2FE}T}\, \diffSyncWalkLaw_{z}\Bigl(\bigcap_{i\in\Hrange\cap I_T} \calD_{i}^\comp \Bgiven \calQ_{m}(w) \Bigr),
	\end{align*}
	where the inequality is obtained by fixing a trajectory on $[0,m]\setminus I_T$ for every realization of $\bigcap_{i\in\Hrange\cap I_T} \calD_{i}^\comp$. That $\Cr{step2FE}\geq 0$ can be chosen uniformly over those realizations follows from the fact that the walk's displacement is constrained by the cone property (see Property~\ref{DRW-ConeProp} in Section~\ref{sec:DiffRW}).
	We control the remaining probability via a union bound:
	\begin{align*}
		\diffSyncWalkLaw_{z}\Bigl(\bigcup_{i\in\Hrange\cap I_T} \calD_{i} &\Bgiven \calQ_{m}(w) \Bigr)
		\leq
		\diffSyncWalkExp_{z}\Bigl[\sum_{i=T+1}^{m-T}\I_{\calD_i} \Bgiven \calQ_{m}(w) \Bigr]\\
		&=
		\sum_{i=T+1}^{m-T} \sum_{\ell=1}^{m-i} \diffSyncWalkLaw_{z} \bigl( i\in\Hrange, \norm{\diffSyncWalk_{t_i}^{\perp}} < 2\delta\ell, \check{X}_{t_i+1}^{\parallel}=\ell \bgiven \calQ_{m}(w) \bigr)\\
		&=
		\sum_{i=T+1}^{m-T} \sum_{\ell=1}^{m-i} \sum_{\norm{u}\leq 2\delta\ell} \sum_{\norm{v}\leq 4\delta\ell}
		\diffSyncWalkLaw_{z} \bigl( \calQ_i(u), \check{X}_{t_i+1}=(\ell,v-u) \bgiven \calQ_{m}(w) \bigr).
	\end{align*}
	We now will bound the sum over $i\leq m/2$, the other half is done in the same fashion.
	Let us write
	\[
		\Phi_d(T) =
		\begin{cases}
			T^{-1/2} 		& \text{if } d=2,\\
			(\log T)^{-1}	& \text{if } d=3,\\
			T^{-(d-3)/2}	& \text{if } d\geq 4.
		\end{cases}
	\]
	Using the fact that the increments \(\check{X}_k\) have exponential tails, Lemma~\ref{lem:UBnoReturn} and~\eqref{eq:LBqxy}, we obtain (with the convention that \(\phi_d(0)=1\)):
	\begin{align*}
		\sum_{i=T+1}^{m/2} &\sum_{\ell=1}^{m-i} \sum_{\norm{u}\leq 2\delta\ell} \sum_{\norm{v}\leq 4\delta\ell}
		\diffSyncWalkLaw_{z}(\calQ_i(u), \check{X}_{t_i+1}=(\ell,v-u) \given \calQ_{m}(w))\\
		&=
		\sum_{i=T+1}^{m/2} \sum_{\ell=1}^{m-i} \sum_{\norm{u}\leq 2\delta\ell} \sum_{\norm{v}\leq 4\delta\ell}
		\frac{q_i(z,u)\, \diffSyncWalkLaw(\check{X}_{t_{i+1}}=(\ell,v-u))\, q_{m-i-\ell}(v,w)}{q_m(z,w)}\\
		&\leq
		c \sum_{i=T}^{m/2} \sum_{\ell=1}^{m-i} \sum_{\norm{u}\leq 2\delta\ell} \sum_{\norm{v}\leq 4\delta\ell}
		(1+\norm{u})^{d+1} (1+\norm{v})^{d+1} \frac{\phi_d(i)e^{-c\ell}\phi_d(m-i-\ell)}{\phi_d(m)}\\
		&\leq
		c \sum_{i=T}^{m/2} \phi_d(i) \sum_{\ell=1}^{m-i}
		\ell^{4d} e^{-c\ell} \frac{\phi_d(m-i-\ell)}{\phi_d(m)}\\
		&\leq
		c \sum_{i=T}^{\infty} \phi_d(i)
		\leq
		c\, \Phi_d(T)
		\leq
		\frac{1}{4},
	\end{align*}
	provided $T$ be chosen large enough as a function of $z,w$ (since \(c=c(z,w)\)).

	Repeating this argument for the other half of the sum yields the same bound; we thus have (with the choice of $T$ as mentioned before)
	\begin{align*}
		\diffSyncWalkLaw_{z}\Bigl(\bigcap_{i\in\Hrange\cap[0,m]} \calD_{i}^\comp \Bgiven \calQ_{m}(w) \Bigr)
		\geq
		\tfrac{1}{2} e^{-\Cr{step2FE}T},
	\end{align*}
	which concludes the proof.
\end{proof}

%
%

\appendix

\section{Random walk estimates}\label{app:RW}

In this section, we provide some random walk estimates that are needed in the paper. Since we are already losing multiplicative constants when reducing the analysis to non-intersecting random walks, we do not try to get sharp asymptotics, but prefer to provide instead a short self-contained analysis. It should however be noted that the approach in~\cite{Jain+Pruitt-1972} and~\cite{Uchiyama-2011} can be adapted to our present setting and would yield sharp asymptotics (and even information on higher-order corrections). Set
\[
	\psi_d(n) =
	\begin{cases}
		n^{-1} 			& \text{when }d=2,\\
		\log(1+n)^{-2} 	& \text{when }d=3.
	\end{cases}
\]
We use the notations introduced in Section~\ref{sec:NotationsRW}.
\begin{proposition}\label{pro:AppRW}
Let \((\diffSyncWalk_n)_{n\geq 0}\) denote the random walk introduced in Section~\ref{sec:RWrep}.
Then, when \(d\in\{2,3\}\), there exist \(0< c_-\leq c_+ < \infty\) such that, for all \(n\geq 1\),
\[
c_- \psi_d(n) \leq \diffSyncWalkLaw_{0}\bigl( \calQ_n(0) \bgiven \calP_n(0) \bigr) \leq c_+ \psi_d(n) .
\]
Moreover, for any \(d\geq 4\), there exists \(c_-(d)>0\) such that, for all \(n\geq 1\),
\[
\diffSyncWalkLaw_{0}\bigl( \calQ_n(0) \bgiven \calP_n(0) \bigr) \geq c_-(d).
\]
\end{proposition}

\medskip
The remainder of this appendix is devoted to a proof of Proposition~\ref{pro:AppRW}. The arguments used below are heavily inspired by those in~\cite{Jain+Pruitt-1972,Dvoretzky+Erdos-1951}.
To shorten notations, we set, for \(n\in\bbZ_{\geq 0}\),
\[
u_n = p_n(0,0),\qquad
f_n = q_n(0,0).
\]
Note that the quantity we want to control can then be expressed as \(f_n/u_n\).

\subsection*{Upper bound on \(p_n(x,y)\)}
It follows from the local limit theorem (see~\cite{Ioffe-2015}) that there exists \(C_p^+\) such that, for all \(x,y\in\bbZ^{d-1}\), all \(n\geq 1\) and all \(d\geq 2\),
\begin{equation}\label{eq:UBp}
p_n(x,y) \leq C_p^+ n^{-(d-1)/2}.
\end{equation}

\subsection*{Lower bound on \(p_n(x,y)\)}
It follows from the local theorem (see~\cite{Ioffe-2015}) that, for any \(c>0\), there exist \(C_p^->0\) such that, for all \(x,y\in\bbZ^{d-1}\) satisfying \(\norm{x-y}\leq c\sqrt{n}\), all \(n\geq 1\) and all \(d\geq 2\),
\begin{equation}\label{eq:LBp}
p_n(x,y) \geq C_p^- n^{-(d-1)/2}.
\end{equation}

\subsection*{Upper bound on \(r_n\)}
Since
\(
\sum_{m=0}^n u_m r_{n-m} = 1
\)
and the sequence \((r_m)_{m\geq 1}\) is decreasing, we have
\[
1 \geq r_n \sum_{m=0}^n u_m.
\]
It then follows from~\eqref{eq:LBp} (and \(u_0=1\)) that there exists \(C_r^+\) such that, for all \(n\geq 1\),
\begin{equation}\label{eq:UBr}
r_n \leq C_r^+
\begin{cases}
n^{-1/2}		&	\text{when }d=2,\\
(1+\log n)^{-1}	&	\text{when }d=3.
\end{cases}
\end{equation}

\subsection*{Lower bound on \(r_n\)}
We start with the case \(d=3\).
Let \(C>1\) be a large constant (to be chosen later). Since
\(
\sum_{m=0}^{Cn} u_m r_{Cn-m} = 1
\)
and the sequence \((r_m)_{m\geq 1}\) is decreasing (and bounded by \(1\)), we have, using~\eqref{eq:UBp},
\[
r_n \sum_{m=0}^{Cn-n} u_m
\geq
1 - \sum_{m=Cn-n+1}^{Cn} u_m
\geq
1 - C_p^+ \sum_{m=Cn-n+1}^{Cn} m^{-1}
\geq
\frac 12,
\]
once \(C\) is chosen large enough. Since, again by~\eqref{eq:UBp},
\(
\sum_{m=0}^{Cn-n} u_m \leq C_p^+\log(Cn)
\),
we conclude that there exists \(C_r^->0\) such that, for all \(n\geq 1\),
\begin{equation}\label{eq:LBr}
r_n \geq C_r^- (1+\log n)^{-1}.
\end{equation}
Let us now turn to the case \(d=2\). Proceeding similarly as above, we write
\[
r_n \sum_{m=0}^{Cn-n} u_m
\geq
1 - \sum_{m=Cn-n+1}^{Cn} u_m r_{Cn-m}
\geq
1 - C_p^+ ((C-1)n)^{-1/2} \sum_{m=0}^{n-1} r_m.
\]
Using~\eqref{eq:UBr} (and \(r_0=1\)), we see that the last term can again be made smaller than \(1/2\) by choosing \(C\) large enough. Of course,
\(
\sum_{m=0}^{Cn-n} u_m
\leq
C_p^+ (Cn)^{1/2}
\)
by~\eqref{eq:UBp}. It thus follows that one can find \(C_r^->0\) such that, for all \(n\geq 1\),
\begin{equation}\label{eq:LBr2d}
r_n\geq C_r^- n^{-1/2}.
\end{equation}

\subsection*{Upper bound on \(f_n\)}
Set \(I=\{\frac n4, \dots, \frac n3\}\).
Let
\(
\vect\tau = \inf\setof{k\geq 0}{S_k^\parallel \geq n/4}
\)
and
\(
\tcev{\tau} = \sup\setof{k\geq 0}{S_k^\parallel \leq 3n/4}
\).
Now, observe that, since the increments of \(\diffSyncWalk\) have exponential tails, there exists \(c>0\) such that
\begin{multline*}
f_n
\leq
e^{-cn} +
\sum_{m,m'\in I}
\sum_{u,u'\neq 0}
\diffSyncWalkLaw_{0}
\bigl(
	S_{\vect\tau}^\perp = u,  S_{\vect\tau}^\parallel = m,
	S_{\tcev\tau}^\perp = u', S_{\tcev\tau}^\parallel = n-m',
	S_{\tau_0^\perp}^\parallel = n
\bigr) .
\end{multline*}
Applying twice the strong Markov property (once for the walk itself, once for the time-reversed walk), we can write
\begin{multline*}
\diffSyncWalkLaw_{0}
\bigl(
	S_{\vect\tau}^\perp = u,  S_{\vect\tau}^\parallel = m,
	S_{\tcev\tau}^\perp = u', S_{\tcev\tau}^\parallel = n-m',
	S_{\tau_0^\perp}^\parallel = n
\bigr)
=\\
\diffSyncWalkLaw_{0}
\bigl(
	S_{\vect\tau}^\perp = u,  S_{\vect\tau}^\parallel = m, \tau_0^\perp > \vect\tau
\bigr)
\diffSyncWalkLaw_{0}
\bigl(
	S_{\vect\tau}^\perp = u',  S_{\vect\tau}^\parallel = m', \tau_0^\perp > \vect\tau
\bigr)\\
\times q_{n-m'-m}(u,u').
\end{multline*}
Now, on the one hand, \eqref{eq:UBp} implies that
\[
q_{n-m'-m}(u,u')
\leq
p_{n-m'-m}(u,u')
\leq
C_p^+ (n/3)^{-(d-1)/2}.
\]
On the other hand, by~\eqref{eq:UBr},
\begin{align*}
\sum_{m\in I}\sum_{u\neq 0}
\diffSyncWalkLaw_{0}
\bigl(
	S_{\vect\tau}^\perp = u,  S_{\vect\tau}^\parallel = m, \tau_0^\perp > \vect\tau
\bigr)
&\leq
r_{n/4}
&\leq
C_r^+
\begin{cases}
(n/4)^{-1/2}					& \text{if } d=2,\\
\bigl(\log(1+n/4)\bigr)^{-1}	& \text{if } d=3.
\end{cases}
\end{align*}
Of course, the same applies to the sum over \(m',u'\). Overall, we conclude that there exists \(C_f^+\) such that, for all \(n\geq 1\),
\begin{equation}\label{eq:UBf}
f_n \leq C_f^+
\begin{cases}
n^{-3/2}							&	\text{when }d=2,\\
n^{-1}\bigl(\log (1+n)\bigr)^{-2}	&	\text{when }d=3.
\end{cases}
\end{equation}

\subsection*{Lower bound on \(f_n\)}

\begin{figure}[t]
\centering
\includegraphics[width=.8\textwidth]{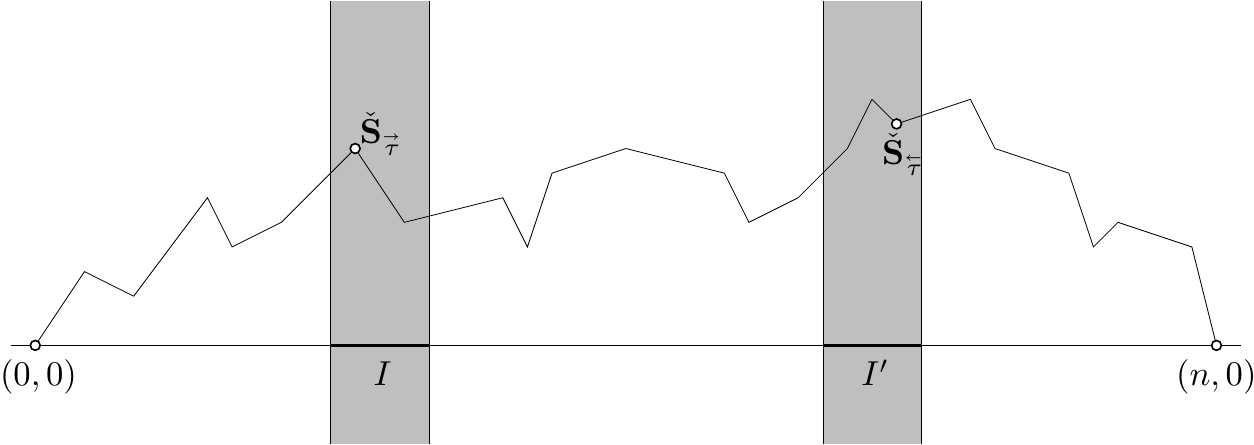}
\caption{The type of decomposition used in the derivations of the upper and lower bounds on \(f_n\). Various choices are made for the intervals \(I\) and \(I'\).}
\end{figure}

We start with the case \(d=3\).
Our goal is to prove that there exists \(C_f^- > 0\) such that, for all \(n\geq 1\),
\begin{equation}\label{eq:LBf3d}
f_n \geq C_f^- n^{-1} \bigl( \log (1+n) \bigr)^{-2}.
\end{equation}
First, let us set \(M=[n(\log n)^{-3}]\) and \(I=\{M, \dots, 2M\}\).
Similarly as we did for the upper bound, let us introduce
\(
\vect\tau = \inf\setof{k\geq 0}{S_k^\parallel \geq M}
\)
and
\(
\tcev{\tau} = \sup\setof{k\geq 0}{S_k^\parallel \leq n-M}
\).
We can then write
\begin{align*}
f_n
&\geq
\sum_{m,m'\in I}
\sum_{u,u'\neq 0}
\diffSyncWalkLaw_{0}
\bigl(
	S_{\vect\tau}^\perp = u,  S_{\vect\tau}^\parallel = m,
	S_{\tcev\tau}^\perp = u', S_{\tcev\tau}^\parallel = n-m',
	S_{\tau_0^\perp}^\parallel = n
\bigr) \\
&=
\sum_{m,m'\in I}
\sum_{u,u'\neq 0}
\vect{q}_m(u) \, q_{n-m'-m}(u,u') \, \tcev{q}_{m'}(u') .
\end{align*}
where we have introduced
\(
\vect{q}_m(u)
=
\diffSyncWalkLaw_{0} \bigl(
	S_{\vect\tau}^\perp = u,  S_{\vect\tau}^\parallel = m, \tau_0^\perp > \vect\tau
\bigr)
\) and
\(
\tcev{q}_{m'}(u')
=
\diffSyncWalkLaw_{0}
\bigl(
	S_{\vect\tau}^\perp = u',  S_{\vect\tau}^\parallel = m', \tau_0^\perp > \vect\tau
\bigr)
\).
The next observation is that
\begin{align*}
0\leq p_\ell(u,u') - q_\ell(u,u')
&=
\diffSyncWalkLaw_{u}(\tau_0^\perp \leq \ell, \calP_\ell(u'))\\
&\leq
\sum_{r=1}^{\ell/2} q_r(u,0) p_{\ell-r}(0,u') + \sum_{r=\ell/2}^{\ell-1} p_r(u,0) q_{\ell-r}(0,u').
\end{align*}
Consequently, writing \(N=N(n,m,m')=n-m-m'\),
\begin{align*}
\Bigl| \sum_{m,m'\in I}
\sum_{u,u'\neq 0}
&\vect{q}_m(u) \, \bigl(q_{N}(u,u')-p_{N}(u,u')\bigr) \, \tcev{q}_{m'}(u') \Bigr|\\
&\leq
\sum_{m,m'\in I}
\sum_{u,u'\neq 0}
\sum_{r=1}^{N/2}
\vect{q}_m(0,u) q_r(u,0) p_{N-r}(0,v) \tcev{q}_{m'}(u',0)\\
&\qquad+
\sum_{m,m'\in I}
\sum_{u,u'\neq 0}
\sum_{r=N/2}^{N-1}
\vect{q}_m(0,u) p_r(u,0) q_{N-r}(0,v) \tcev{q}_{m'}(u',0).
\end{align*}
By symmetry, it suffices to bound the first sum. First, by~\eqref{eq:UBp} and the fact that \(r\leq N/2\),
\begin{gather*}
p_{N-r}(0,v) \leq 2C_p^+ N^{-1} \leq 2 C_p^+ (n-4M)^{-1} \leq C n^{-1}.
\end{gather*}
Second, by~\eqref{eq:UBf},
\[
\sum_{m\in I} \sum_{u\neq 0} \sum_{r=1}^{N/2} \vect{q}_m(0,u) q_r(u,0)
\leq
\sum_{r=1}^{n} f_{M+r}
\leq
\frac{C_f^+}{(\log M)^2} \sum_{r=1}^n (M+r)^{-1}
\leq
C \frac{\log\log n}{(\log n)^2}.
\]
Finally, by~\eqref{eq:UBr}
\begin{equation}\label{eq:sumum}
\sum_{m'\in I} \sum_{u'\neq 0} \tcev{q}_{m'}(u',0)
\leq
r_M
\leq
C_r^+ (\log M)^{-1}
\leq
C (\log n)^{-1}.
\end{equation}
We conclude that
\[
\Bigl| \sum_{m,m'\in I} \sum_{u,u'\neq 0}
\vect{q}_m(u) \, \bigl(q_{N}(u,u')-p_{N}(u,u')\bigr) \, \tcev{q}_{m'}(u') \Bigr|\\
\leq
C \frac{\log\log n}{n(\log n)^3},
\]
which is negligible in view of our target estimate.

Now, by the local CLT~\cite{Ioffe-2015}, there exists \(c\) such that
\[
\bigl| p_\ell(u,u') - \frac{c}{\ell} \bigr| \leq \frac{\epsilon}{\ell} + O(\ell^{-2}\norm{u'-u}^2),
\]
uniformly in \(u,u'\) for all \(\ell\) large enough. Using \(\norm{u'-u}^2 \leq 2\norm{u}^2 + 2\norm{u'}^2\), \eqref{eq:sumum} and
\[
\sum_{m\in I} \sum_{u\neq 0} \vect{q}_{m}(u,0) \norm{u}^2
\leq
\diffSyncWalkExp_0(\norm{\diffSyncWalk^\perp_{\tau_M^\parallel}}^2)
\leq
C M ,
\]
we deduce that
\[
n^{-2} \sum_{m,m'\in I} \sum_{u,u'\neq 0} \vect{q}_m(0,u) \norm{u'-u}^2 \tcev{q}_{m'}(u',0)
\leq
C n^{-1}(\log n)^{-3},
\]
which is also negligible. \eqref{eq:LBf3d} thus follows from
\[
\sum_{m,m'\in I} \sum_{u,v\neq 0} \vect{q}_m(0,u) N^{-1} \tcev{q}_{m'}(u',0) \geq n^{-1} (1-e^{-c M})^2 (r_M)^2 \geq \tfrac12 (C_r^-)^2 n^{-1}(\log n)^{-2},
\]
where we used~\eqref{eq:LBr} and the exponential tails of the random walk increments.

\medskip
The same argument applies when \(d\geq 4\). Indeed, proceeding as before, but with \(M\) being now a large constant independent of \(n\), we get
\[
\Bigl| \sum_{m,m'\in I} \sum_{u,u'\neq 0}
\vect{q}_m(u) \, \bigl(q_{N}(u,u')-p_{N}(u,u')\bigr) \, \tcev{q}_{m'}(u') \Bigr|\\
\leq
C n^{-(d-1)/2} \sum_{r=M+1}^\infty f_r.
\]
Fix \(\epsilon>0\). Since \(\sum_{r\geq 1} f_r < 1\), the above abound is smaller than \(\epsilon n^{-(d-1)/2}\), provided \(M>M_0(\epsilon)\).
Then, again by the local CLT,
\[
\bigl| p_\ell(u,u') - c\ell^{-(d-1)/2} \bigr|
\leq
\frac{\epsilon}{\ell^{(d-1)/2}} + O(\ell^{-(d+1)/2}\norm{u'-u}^2),
\]
uniformly in \(u,u'\) for all \(\ell\) large enough. Proceeding as above, the contribution of the second term is seen to be of order \(n^{-(d+1)/2}\) and thus negligible. Therefore, since
\[
\sum_{m,m'\in I} \sum_{u,u'\neq 0} \vect{q}_m(0,u) N^{-(d-1)/2} \tcev{q}_{m'}(u',0) \geq C n^{-(d-1)/2}
\]
for some constant \(C>0\), we conclude that there exists \(C_f^->0\) such that
\begin{equation}\label{eq:LBf4d}
f_n \geq C_f^- n^{-(d-1)/2}.
\end{equation}

\medskip
Let us finally turn to the case \(d=2\), for which we need to proceed differently.
Fix \(x,y > 0\) such that \(\diffSyncWalkLaw_0(\diffSyncWalk_1=(x,y))=c>0\). Clearly,
\begin{align*}
f_n
&\geq
\sum_{k=2}^n \diffSyncWalkLaw_0( \diffSyncWalk_k=(n,0), \diffSyncWalk_1=(x,y), \diffSyncWalk_{k-1}=(n-x,y), \diffSyncWalk_j^\perp \geq y\; \forall 2\leq j\leq k-2)\\
&\geq
c^2 \sum_{k\geq 1} \diffSyncWalkLaw_0(\diffSyncWalk_k=(n-2x,0), \diffSyncWalk_j^\perp \geq 0\; \forall 1\leq j\leq k-1).
\end{align*}
Consider a trajectory \((\diffSyncWalk_j(\omega))_{j=0}^k\) be such that \(\diffSyncWalk_0(\omega)=0\) and \(\diffSyncWalk_k(\omega)=(n-2x,0)\). Denote by \(\check{X}_j(\omega)=\diffSyncWalk_{j+1}(\omega)-\diffSyncWalk_{j}(\omega)\) the corresponding increments.
Let \(j_0 = \min\setof{i\geq 0}{\diffSyncWalk_i^\perp(\omega)=\min_{0\leq j\leq k} \diffSyncWalk_k^\perp(\omega)}\). Define a new trajectory by setting \(\hatSyncWalk_0(\omega)=0\) and, for \(1\leq j\leq k\),
\[
\hatSyncWalk_j(\omega) = \sum_{i=0}^{j-1} \check{X}_{(j_0+i) \text{ mod } k}(\omega) .
\]
Observe that \(\hatSyncWalk_0=0\), \(\hatSyncWalk_k = (n,0)\) and \(\hatSyncWalk_i\geq 0\) for all \(1\leq i\leq k-1\), and that the transformation is measure-preserving. Therefore,
\[
\diffSyncWalkLaw_0(\diffSyncWalk_k=(n-2x,0), \diffSyncWalk_j^\perp \geq 0\; \forall 1\leq j\leq k-1) \geq k^{-1} \diffSyncWalkLaw_0(\diffSyncWalk_k=(n-2x,0))
\]
and, therefore, using~\eqref{eq:LBp},
\begin{equation}\label{eq:LBf2d}
f_n \geq c^2 n^{-1} u_{n-2x} \geq C_f^- n^{-3/2},
\end{equation}
for some \(C_f^->0\).

%
%

\bibliographystyle{plain}
\bibliography{OV17}

\end{document}